\documentclass[11pt]{article}

\usepackage[paperwidth=8.5in,paperheight=11.0in,
  left=0.7in,right=1.0in,top=1.0in,bottom=1.3in,
  includefoot,heightrounded]{geometry}

\usepackage{textcomp}
\usepackage{lineno}
\usepackage{ifthen}
\usepackage{cite}
\newboolean{PrintVersion}
\setboolean{PrintVersion}{false} 
\usepackage{lipsum}

\usepackage{multirow}  

\usepackage{xcolor}
\usepackage{amsmath,amssymb,amstext}
\usepackage{centernot}
\usepackage[pdfpagelabels=true,pagebackref=true,bookmarks]{hyperref} 

\hypersetup{
    colorlinks=true,        
    linkcolor=blue,         
    citecolor=green,        
    filecolor=magenta,      
    urlcolor=cyan           
}

\ifx\HCode\UnDef\else\hypersetup{tex4ht}\fi

\let\origdoublepage\cleardoublepage
\newcommand{\clearemptydoublepage}{%
\clearpage{\pagestyle{empty}\origdoublepage}}
\let\cleardoublepage\clearemptydoublepage

\usepackage{float}
\restylefloat{table}
\usepackage{bbm}
\usepackage{exscale}
\usepackage{tabularx}
\usepackage{syntonly}
\usepackage{lineno}

\usepackage{algorithm}
\usepackage{algorithmic}
\usepackage{amssymb}
\usepackage{amsmath}
\usepackage{amsthm}
\usepackage{latexsym}
\usepackage{makeidx}   
\usepackage{longtable}
\usepackage{comment}
\usepackage{caption}
\usepackage{subcaption}

\makeindex

\usepackage{enumerate}
\usepackage{longtable}
\usepackage{xfrac}

\usepackage{algorithmic}
\usepackage{setspace}  

\usepackage[noabbrev]{cleveref}

\let\oldindex\index
\renewcommand*{\index}[1]{\oldindex{#1}\ignorespaces}

\newtheorem{theorem}{Theorem}[section]
\newtheorem{lemma}[theorem]{Lemma}
\newtheorem{definition}[theorem]{Definition}

\newtheorem{example}[theorem]{Example}

\newtheorem{prop}[theorem]{Proposition}

\newtheorem{thm}[theorem]{Theorem} 
\newtheorem{cor}[theorem]{Corollary}
\newtheorem{remark}[theorem]{Remark}

\crefname{thm}{Theorem}{Theorems}
\Crefname{thm}{Theorem}{Theorems}
\crefname{problem}{Problem}{Theorems}
\Crefname{problem}{Problem}{Theorems}
\crefname{conjecture}{Conjecture}{Theorems}
\Crefname{conjecture}{Conjecture}{Theorems}
\crefname{proposition}{Proposition}{Propositions}
\Crefname{proposition}{Proposition}{Propositions}
\crefname{prop}{Proposition}{Propositions}
\Crefname{prop}{Proposition}{Propositions}
\crefname{cor}{Corollary}{Corollaries}
\crefname{lem}{Lemma}{Lemmas}
\Crefname{lem}{Lemma}{Lemmas}
\theoremstyle{definition}
\crefname{assump}{Assumption}{assumptions}
\Crefname{assump}{Assumption}{Assumptions}
\crefname{problem}{Problem}{problems}
\Crefname{problem}{Problem}{Problems}
\crefname{definition}{Definition}{definitions}
\Crefname{definition}{Definition}{Definitions}
\crefname{defn}{Definition}{definitions}
\Crefname{defn}{Definition}{Definitions}
\crefname{remark}{Remark}{Remarks}
\Crefname{remark}{Remark}{Remarks}
\crefname{rmk}{Remark}{Remarks}
\Crefname{rmk}{Remark}{Remarks}
\crefname{example}{Example}{Examples}
\Crefname{example}{Example}{Examples}
\crefname{align}{}{}
\Crefname{align}{}{}
\crefname{equation}{}{}
\Crefname{equation}{}{}

\DeclareMathOperator{\trace}{{trace}}

\DeclareMathOperator{\argmax}{{argmax}}

\DeclareMathOperator{\blkdiag}{{blkdiag}}
\DeclareMathOperator{\Diag}{{Diag}}

\DeclareMathOperator{\relint}{{relint}}

\DeclareMathOperator{\rank}{{rank}}

\DeclareMathOperator{\spann}{{span}}

\DeclareMathOperator{\dist}{{dist}}

\newcommand{\LP}{\textbf{LP}\,}
\newcommand{\LPp}{\textbf{LP}}

\newcommand{\SDP}{\textbf{SDP}\,}
\newcommand{\SDPp}{\textbf{SDP}}

\newcommand{\FR}{\textbf{FR}\,}
\newcommand{\FRp}{\textbf{FR}}

\DeclareMathOperator{\sd}{SD}  
\DeclareMathOperator{\maxsd}{maxSD}  
\DeclareMathOperator{\ips}{IPS}  

\newcommand{\BFS}{\textbf{BFS}\,}
\newcommand{\BFSp}{\textbf{BFS}}

\def\R{\mathbb{R}}

\def\N{\mathbb{N}}

\def\Rn{\mathbb{R}^n}
\def\Rm{\mathbb{R}^m}
\def\Rmn{\R^{m\times n}}

\def\Rnp{\mathbb{R}_+^n}

\def\Rnpp{\mathbb{R}_{++}^n}

\def\Snp{\mathbb{S}^n_+}

\newcommand{\cB}{{\mathcal B}}
\newcommand{\cC}{{\mathcal C}}
\newcommand{\cD}{{\mathcal D}}

\newcommand{\cF}{{\mathcal F}}
\newcommand{\cG}{{\mathcal G}}

\newcommand{\cI}{{\mathcal I}}
\newcommand{\cJ}{{\mathcal J}}
\newcommand{\cK}{{\mathcal K}}

\newcommand{\cN}{{\mathcal N}}

\newcommand{\cP}{{\mathcal P}}

\newcommand{\cR}{{\mathcal R}}
\newcommand{\cS}{{\mathcal S}}

\newcommand{\cV}{{\mathcal V}}

\newcommand{\bbm}{\begin{bmatrix}}
\newcommand{\ebm}{\end{bmatrix}}
\newcommand{\bem}{\begin{pmatrix}}
\newcommand{\eem}{\end{pmatrix}}

\def\<{\langle}
\def\>{\rangle}

\newcommand{\textdef}[1]{\index{#1}\textit{#1}}

\usepackage{mathtools}  

\numberwithin{algorithm}{section}
\numberwithin{equation}{section}
\numberwithin{figure}{section}
\numberwithin{table}{section}

\DeclareMathOperator{\supp}{supp}

\DeclareMathOperator{\face}{face}

\DeclareMathOperator{\nul}{null}
\DeclareMathOperator{\range}{range}

\usepackage{mathrsfs}

\begin{document}



	\title{
Revisiting  Degeneracy, Strict Feasibility, 
Stability,\\ in\\ Linear Programming}
	             \author{
 \href{}{Haesol Im}\thanks{
Department of Combinatorics and Optimization
	Faculty of Mathematics, University of Waterloo, Waterloo,
	Ontario, Canada N2L 3G1;
	\url{}}
	\and
	\href{http://www.math.uwaterloo.ca/~hwolkowi/}
	{Henry Wolkowicz}%
	\thanks{Department of Combinatorics and Optimization
	Faculty of Mathematics, University of Waterloo, Waterloo,
	Ontario, Canada N2L 3G1; Research supported by The Natural
	Sciences and Engineering Research Council of Canada;
	\url{www.math.uwaterloo.ca/\~hwolkowi}.
	}
	}

\maketitle

\begin{abstract}
Currently, the simplex method and the interior point method are indisputably
the most popular algorithms for solving linear programs, \LPp s.
Unlike general conic programs, \LPp s with a finite optimal
value do not require strict feasibility in order to establish strong
duality. Hence strict feasibility is seldom a concern, even though
strict feasibility is equivalent to stability and a compact dual optimal set.
This lack of concern is also true for other types of 
degeneracy of basic feasible solutions in \LPp.
In this paper we discuss that the specific degeneracy that arises from
lack of strict feasibility necessarily 
causes difficulties in both simplex and interior point methods. In
particular, we show that the lack of strict feasibility implies that every basic
feasible solution, \BFSp, is degenerate; thus conversely, the existence of a nondegenerate \BFS implies that strict feasibility (regularity) holds.
We prove the results using facial reduction and simple linear algebra.
In particular, the facially reduced system reveals the 
implicit non-surjectivity of the linear map of the equality constraint system.
As a consequence, we emphasize that facial reduction involves two steps
where, the first guarantees strict feasibility, and the second recovers 
full row rank of the constraint matrix. This illustrates the implicit
\emph{singularity} of problems where strict feasibility fails, and also
helps in obtaining new efficient techniques for preproccessing.
We include an efficient preprocessing method that can be performed 
as an extension of phase-I of the two-phase simplex method.
We show that this can be used to avoid the loss of precision for many
well known problem sets in the literature, e.g.,~the NETLIB problem set.
\end{abstract}

	

\tableofcontents
\listoftables
\listoffigures
\listofalgorithms


\section{Introduction}


\index{\LPp, linear program}
\index{linear program, \LPp}

The \textdef{Slater condition} (strict feasibility) is a useful property 
for optimization models to have.
Unlike general conic programs, linear programs (\LPp s) do not require strict
feasibility as a constraint qualification to guarantee strong
duality, and therefore, it is often not discussed. In fact, degeneracy
in general is not considered to be a serious concern in linear
programming.
The Goldman-Tucker Theorem \cite{MR21:633} is related in that it 
guarantees a primal-dual optimal solution satisfying strict 
complementarity $x^*+z^* > 0$ for the standard form \LPp.
However, it does not guarantee the existence of a strictly feasible
primal solution $\hat x > 0$. 
The lack of strict feasibility for an \LP does not seem to cause
problems at first glance, especially when the simplex method is used. 
In this manuscript, we show that the failure of strict feasibility
results in degeneracy problems when simplex-type methods are used. 
More specifically, the lack of strict feasibility inevitably renders
\LPp s degenerate, i.e.,~\emph{every basic feasible solution is
degenerate}.\footnote{Conversely, if we can find 
\emph{one} nondegenerate basic feasible solution, then strict
feasibility holds.} Note that strict feasibility along
with full row rank of the linear constraint is the
\textdef{Mangasarian-Fromovitz} constraint qualification
\cite{MR34:7263}. This is equivalent to a compact dual optimal set and is 
equivalent to stability with respect to perturbations of the
right-hand side.

The simplex method \cite{Dant:63} is one of the most popular and 
successful algorithms for solving linear programs. 
Degeneracy, a zero basic variable, could result in cycling and noncovergence.
There are many anti-cycling rules, see
e.g.,~\cite{MR2072929,MR0069584,MR459599,MR1260019,MR1260009} 
and the references therein.
However, techniques for the resolution of degeneracy often result in 
\textdef{stalling} \cite{ro,MR0056264,MR1885204,MR850380}, i.e.,~result
in taking a large number of iterations before leaving a degenerate point and can even fail to leave with current techniques~\cite{MR2072929}.
Degeneracies are known to cause numerical issues when interior point
methods are used, e.g., \cite{MR94j:90021}. For example, degeneracy
can result in singularity of the Jacobian of the optimality
conditions, and thus also in ill-posedness and
loss of accuracy~\cite{GoWo:04}.
We note that the method most often used in the literature when
converting a problem that has a free variable into standard form, 
is to replace the free variable by the difference of two nonnegative variables.
This results in an unbounded primal optimal set and strict
feasibility failing for the dual problem, i.e.,~from our
work we see that this standard approach changes a
well-posed problem into an ill-posed one.

\index{\FRp, facial reduction} 
Our main results on the degeneracy arising from loss of
strict feasibility are shown using the effective preprocessing tool called 
\textdef{facial reduction, \FRp}. 
For a problem lacking strict feasibility, facial reduction strives to 
formulate an equivalent problem that has a Slater point. 
By examining the facially reduced system, we obtain two results. First,
we show that every basic feasible solution is degenerate when strict
feasibility fails. This leads to an efficient approach for
eliminating variables that are fixed at $0$.
Second, we investigate implicit redundancies as a source of instability arising in problems where strict feasibility fails. We see that the linear
map of the facially reduced system is non-surjective, i.e.,~the original constraints are implicitly redundant.
Finally, we use these results to develop an efficient preprocessing
technique to obtain strict feasibility. This technique is illustrated on
instances from the 
\href{https://www.netlib.org/lp/}{NETLIB} data set.

The contribution of this manuscript is threefold;  
(i) We provide the complete description of the facially reduced system of a linear program and introduce related notions of singularity; (ii) We show that every basic feasible solution of a standard linear program is degenerate when strict feasibility fails;
(iii) We propose and illustrate an efficient preprocessing scheme that can 
be performed as an extension of phase-I of the two-phase simplex method.
This technique allows for eliminating variables fixed at $0$, and thus
regularizing and simplifying the \LPp.

The manuscript is organized as follows.
In \Cref{sec:KnownResults} we present the background and notations.
Included are the notions of degeneracy,
facial reduction and three types of singularity degree. We
then describe what facial reduction tries to achieve.
In \Cref{sec:MainResult} we present our main result and immediate
corollaries, 
as well as the efficient preprocessing method that can be used as an
extension of phase-I of the two-phase simplex method.
In addition, we relate our main result to known results in the
literature, such as distance to infeasibility.
In \Cref{sec:Numerics} we illustrate algorithmic performance of 
interior point methods and the simplex method under the lack of 
strict feasibility. We present our conclusions in \Cref{sec:Conclusion}.

\section{Preliminaries}
\label{sec:KnownResults}

\subsection{Background and Notation}
\label{sec:Background}

\index{$\Rnp$, nonnegative orthant}
\index{nonnegative orthant, $\Rnp$}
\index{$\Rnpp$, positive orthant}
\index{positive orthant, $\Rnpp$}
\index{$\Rmn$, real vector space of $m$-by-$n$ matrices}
\index{real vector space of $m$-by-$n$ matrices, $\Rmn$}
\index{$A_\cI$, submatrix of $A$ with columns in $\cI$}
\index{$\relint$, relative interior}
\index{relative interior, $\relint$}
\index{$\<\cdot,\cdot\>$, inner product}
\index{$\supp$}

We let $\Rn,\Rmn$ be the standard real vector spaces of $n$-coordinates and $m$-by-$n$ matrices, respectively.
We use $\Rnp$ ($\Rnpp$, resp.) to denote the $n$-tuple with nonnegative (positive) entries.
We use $\<\cdot,\cdot\>$ to denote the usual inner product. 
Given a vector $x\in \Rn$, we let $\supp(x)$ to denote the index set $\{i : x_i \ne 0\}$.
Given a matrix $A\in \Rmn$, we adopt the MATLAB notation to denote a
submatrix of $A$. Given a subset $\cI$ of column indices, $A_\cI \in \R^{m\times |\cI|}$ is the submatrix of $A$ that contains the columns of $A$ in $\cI$.
We also use the notation $A_\cI$ to denote $A_\cI$ when the meaning is clear. 
Given a convex set $\cC$, $\relint(\cC)$ denotes the relative interior of the set $\cC$.


\index{$\cF$, feasible region}
\index{feasible region, $\cF$}
\index{$(\cP)$}

Throughout this manuscript, we work with feasible \LPp s in standard form
with finite optimal value:
\[
(\cP) \qquad \textdef{$p^*$}=
          \min_x \left\{ c^Tx \ : \ Ax =b, \ x\ge 0 \right\},
\] 
where $p^*\in \R, A \in \Rmn, b\in \Rm$ and $c \in \Rn$.
We assume that $\rank(A) = m$, i.e., there is no redundant constraint. 
We use $\cF$ to denote the feasible region of ($\cP$)
\begin{equation}
\label{eq:feasibleSet}
\cF = \{x \in \Rn: Ax = b, \ x \ge 0 \}.
\end{equation}

\subsubsection{Degeneracy in \LPp}
Given an index set $\cB \subset \{1,\ldots,n\}, |\cB|=m$,
a point $x \in \cF$ is called a \textdef{basic feasible solution, \BFS}, if 
$A_\cB $ is nonsingular and $x_i = 0, \ \forall i \in \{1,\ldots,n\}\setminus \cB$.
It is well-known that the simplex method iterates from \BFS to \BFSp.
A basic feasible solution 
$x\in \cF$ is \emph{nondegenerate} if $x_i>0, \ \forall i \in \cB$; it
is \emph{degenerate} if $x_i = 0$, for some $i \in \cB$.
It is clear that every basic feasible solution has
at most $m$ positive entries.\footnote{We mainly consider primal
degeneracy here, though everything follows through for dual degeneracy.
In fact, there are clear connections from complementary slackness
between variables positive in every \BFS and dual variables fixed at $0$.}

We partition the index set \textdef{$\{1,\ldots,n\}$} as
\index{degenerate \BFS}
\index{nondegenerate \BFS}
\index{\BFS, basic feasible solution}
\[
\{1,\ldots,n\}= \cI_+ \cup \cI_0, \text{ where } 
\textdef{$\cI_0 := \{ i \,:\, x_i = 0, \forall x\in \cF\}$}
\text{ and } 
\textdef{$\cI_+ = \{1,\ldots,n\} \backslash \cI_0$},
\]
i.e.,~$\cI_0$ denotes the variables \textdef{fixed at $0$}.
Note that fixed variables are identified during preprocessing in the
literature if the upper and lower bounds are equal, 
e.g.,~\cite{MR2020137,ANDERSEN:95:D,Huang04preprocessingand}.
However, the set $\cI_0$ is not as easily identified.


There are in fact several types of degeneracy. Let $\bar x$ be a given
\BFS with basis $\cB$. (Wlog $\cB = \{1,\ldots,m\}$.)
We can write the equivalent canonical form representation of the
feasible set using the basis at $\bar x$:
\begin{equation}
\label{eq:canonform}
\cF = \left\{
x = \begin{pmatrix}x_\cB \cr x_\cN\end{pmatrix} \,:\,
x_\cB = b - A_\cB^{-1}A_\cN x_\cN \geq 0, x_\cN\geq 0
\right\}.
\end{equation}
In this form $x_\cN\in \R_+^{n-m}$, we have $n$ inequality constraints, and
we see that degeneracy is equivalent to having an active
set with cardinality greater than $n-m$. 
This divides into two types corresponding to the sets $\cI_0,\cI_+$,
respectively:
(i) inequalities that
are active in every \BFS and correspond to variables in $\cI_0$ above;
(ii) those that are not active in at least one \BFSp. The
geometry of (i) is clear as there is no Slater point and $\cF$ is a
subset of a face of the nonnegative orthant. For (ii) the geometry is
that some of the constraints are redundant in one of two ways, 
i.e.,~that discarding them does not change the feasible set nor the optimality conditions if $\bar x$ is optimal.

\begin{remark}
\label{rem:addredconstr}
We note that adding redundant constraints is done in
e.g.,~\cite{MR2367063,MR2238662} to show that the central path for 
interior point methods can follow the boundary closely, i.e.,~behave
very poorly. These redundant constraints correspond to a positive
variable in each \BFSp, i.e.,~to an inequality  in~\cref{eq:canonform}
that is never active.  Complementary slackness implies that they
correspond to variables fixed at $0$ in the dual problem, thus
emphasizing that \FR on the dual could avoid some of these difficulties.
\end{remark}

\subsection{Facial Reduction}
\label{sec:FacialReduction}

In this section we describe the concept of facial reduction and  
present the properties that are used to establish the main result. 
We emphasize in this paper that facial reduction for ($\cP$) involves \emph{two}
steps: first, obtain an equivalent problem with strict feasibility;
second, recover full row rank of the constraint matrix. Note that
full row rank is \emph{always} lost during the first step.

Let $K\subset \Rn$ be a convex set.
A convex set $F \subseteq K$ is called a \textdef{face} of $K$, 
denoted $F\unlhd K$, if  for all $y,z \in K$ with $x = \frac{1}{2}(y +z) \in F$, we have  $y,z \in F$. 
Given a convex set $\cC\subseteq K$, the \textdef{minimal face} for $\cC$ is the intersection of all faces containing the set $\cC$.

\begin{prop}\cite[Theorem 3.1.3]{DrusWolk:16}(theorem of the alternative)
\label{prop:Farkas}
For the feasible system of~\cref{eq:feasibleSet}, exactly one of the following statements holds:
\begin{enumerate}
\item There exists $x \in \R^n_{++}$ with $Ax = b$, i.e.,~strict
feasibility holds;
\item There exists $y\in \Rm$ such that 
\begin{equation}
\label{eq:auxsystem}
0\neq z := A^Ty \in  \Rn_+,  \ \text{ and } \  \<b,y\>=0.
\end{equation}
\end{enumerate}
\end{prop}

\Cref{prop:Farkas} gives rise to a process called \emph{facial reduction}.
The \textdef{facial reduction, \FRp}, for an \LP is a process of
identifying the minimal face of $\Rn_+$ containing the feasible
set $\cF=\{x \in \Rnp : Ax =b\}$. 
By finding the minimal face, we can work with a problem that lies in a smaller
dimensional space and that statisfies strict feasibility.
The \FR process, i.e., finding the minimal face,  is usually done by
solving a sequence of auxiliary systems \cref{eq:auxsystem}. 
More details on \FR on general conic problems can be found in
\cite{bw1,bw3,DrusWolk:16,Sremac:2019,permenter2017reduction}.

We now describe how the set $\cF$ (see \cref{eq:feasibleSet}) is represented after \FRp.
Suppose that strict feasibility fails. 
Then \Cref{prop:Farkas} implies that there must exist a nonzero $y\in \Rm$ satisfying 
\begin{equation}
\label{eq:xexposed}
\<x,A^Ty\> = \<Ax,y\> = \<b,y\> = 0 , \ \forall x \in \cF.
\end{equation}
Hence, every $x\in \cF$ is perpendicular to the nonnegative vector $z=A^Ty$.
We call this vector $z=A^Ty$ an \textdef{exposing vector} for $\cF$, and let the cardinality of its support be $s_z=| \{i : z_i>0\}|$.
Then $z = \sum\limits_{j=1}^{s_z} z_{t_j} e_{t_j}$, where  $t_j$ is in increasing order. 
We now have
\[
0=\<z,x\> \ \text{ and } \  x,z\in \Rnp \  \implies \, x_iz_i = 0, \ \forall i,
\]
i.e.,~the positive elements in $z$ identify the corresponding elements
in $x$ that are \textdef{fixed at $0$}.
Then $x = \sum\limits_{j=1}^{n-s_z} x_{s_j} e_{s_j}$, where $s_j$ is in increasing order. 
We define the matrix with unit vectors for columns
\[
V  =\begin{bmatrix}e_{s_1} & e_{s_2} & \ldots & e_{s_{n-s_z}}
    \end{bmatrix} \in \R^{n\times (n-s_z)}.
\]
Then we have
\begin{equation}
\label{eq:setEqauivalence}
\cF = \{x \in \Rn_+ : Ax = b\} =
\{x = Vv \in \Rn :   AVv = b, v \in \R^{n-s_z}_+\} .
\end{equation}
We call this matrix $V \in \R^{n\times (n-s_z)}$ a \textdef{facial range vector}.
The facial range vector restricts the support of all feasible $x$.
We use the identification \cref{eq:setEqauivalence} throughout this manuscript. 
This concludes the first step of \FRp, i.e.,
identifying all the variables that are fixed at
$0$.\footnote{Note that this can be done in one step for linear
programs, i.e.,~the singularity degree for \LP is one. We discuss 
this in~\Cref{sec:FacialReduction}.}

\index{$s_z$, support of exposing vector for $\cF$}
\index{support of exposing vector for $\cF$, $s_z$}

It is known that every facial reduction step results in at least one constraint
being redundant, see e.g., \cite{bw3}, \cite[Lemma 2.7]{ImWolk:21}, and \cite[Section 3.5]{Sremac:2019}.
For completeness we now include a short proof tailored to 
\LPp, see~\Cref{lemma:ConstrRedundant}.
\begin{lemma}
\label{lemma:ConstrRedundant}
Consider the facially reduced feasible set
\begin{equation*}
\label{eq:setEqauivalenceV}
\cF_r = \left\{v :  AVv = b, v \in \R^{n-s_z}_+\right\} .
\end{equation*}
Then at least one linear constraint of the \LP is redundant.
\end{lemma}

\index{$I$, the identity matrix}
\begin{proof}
Let $z=A^Ty$ be the exposing vector satisfying the auxiliary
system \cref{eq:auxsystem}. And let $V$ be a facial range vector
induced by $z$. Then
\begin{equation}
\label{eq:redunProof}
0 = V^Tz = V^T A^Ty  = (AV)^Ty =\sum_{i=1}^m y_i ((AV)^T)_i. 
\end{equation}
Since $y\in \Rm$ is a nonzero vector, the rows of $AV$ are linearly dependent.
\end{proof}
We now see the result of the full two-step facial reduction process,
i.e.,~we get a constraint matrix of full row rank:
\begin{equation}
\label{eq:setEqauivalencefullrow}
\cF = \{x \in \Rn_+ : Ax = b\} =
\{x = Vv \in \Rn :   P_{\bar m} AVv = P_{\bar m} b, \ v \in \R^{n-s_z}_+\},
\end{equation}
where \textdef{$P_{\bar m} : \Rm \to \R^{\bar m}$}, 
\textdef{$\bar m ={\rank(AV)}$}, is the simple 
projection that chooses the linearly independent rows of $AV$.
This concludes the second step of \FRp, i.e., guaranteeing the full rank.
We include a graphical illustration of the two-step \FR process; see \Cref{fig:twostepFRpicture}.
\begin{figure}[h]
\centering
\includegraphics[height=3.5cm]{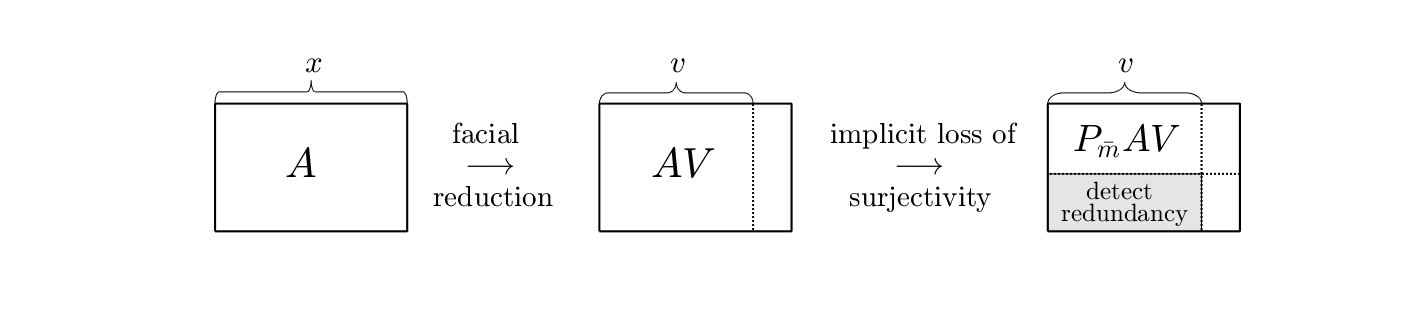}
\vspace{-0.5cm}
\caption{A graphical illustration of the two-step facial reduction}
\label{fig:twostepFRpicture}
\end{figure}

\index{$P_{\bar{m}}AV$}
\index{$P_{\bar{m}}b$}
\index{distance to infeasibility}

For a general conic problem, such as semidefinite programs (\SDPp), 
the facial reduction iterations do not necessarily end in one iteration; see
\cite{Sremac:2019,SWW:17,ScTuWonumeric:07}.
And there is a special name for the minimum length of \FR iterations.
\begin{definition}[{\cite[Sect. 4]{S98lmi}}]
\label{def:SD_definition}
Given a spectrahehedron $\cS$ in a closed convex cone $\cK$, the \textdef{singularity degree, $\sd(\cS)$}
of $\cS$ is the \underline{smallest} 
number of facial reduction iterations for finding $\face(\cS, \cK)$, the minimal face of $\cK$ containing $\cS$.
\end{definition}

It is known that \FR for \LPp s can be done in one iteration, i.e., $\sd(\cF)\le 1$; see [22,
Theorem 4.4.1]. 
\Cref{prop:Farkas} and \Cref{lemma:ConstrRedundant} imply that \emph{any} solution to the system~\eqref{eq:auxsystem}  gives rise to a strict reduction in the number of variables and the number of equality constraints.
This gives rise to the following two novel notions of singularity.

\begin{definition}
\label{def:ips}
Let $K\subseteq \Rn$ be a closed convex cone with corresponding
feasible set $\cS = \{x\in K \,:\, Ax=b\}$ and facially reduced feasible
set $\{v\in PK \,:\, (PA\cV)(v)=Pb, v\in \R^r\}$, where $PA\cV$ is
onto $\R^{m_r}$ and $PK$ is the cone defined over the smaller dimensional space. Then the 
\textdef{implicit problem singularity, $\ips(\cS)= m-m_r$}.
\\Moreover, the \textdef{max-singularity degree} of
$\cS$, denoted $\maxsd(\cS)$, is the \underline{largest} number of 
nontrivial facial reduction iterations for finding $\face(\cS,K)$. 
\end{definition}
\index{$\ips(\cS)= m-m_r$, implicit problem singularity}
\index{$\maxsd(\cS)$, largest number nontrivial facial reduction steps}
\index{largest number nontrivial facial reduction steps, $\maxsd(\cS)$}

The singularity degree is used in~{\cite[Sect. 4]{S98lmi}} for providing
a H\"older regularity constant for semidefinite programs. This is then
used in~\cite{DrusLiWolk:14} to derive a convergence rate for
alternating projection methods for semidefinite programs.
Note that $\maxsd(\cS)$ can be a larger lower bound of $\ips(\cS)$ 
than $\sd(\cS)$, since at
least one linear constraint becomes redundant at each \FR iteration.
The effect on ill-conditioning
of larger values of $\ips$ is seen empirically in 
\Cref{sect:SVDips}.\footnote{\Cref{def:ips} can be used to 
strengthen the upper bound
on the rank of \SDP solutions in~\cite{ImWolk:21}, i.e.,~we get
$t(r) \le m-\ips(\cS) \le m-\maxsd(\cS) \le m-\sd(\cS) \le m$, 
where $t(r)$ is the triangular number of the rank $r$.
}

\subsubsection{Preprocessing in \LP}
An essential step for simplex and interior point methods is
preprocessing, see
e.g.,~\cite{MR2020137,ANDERSEN:95:D,MR1478041,Huang04preprocessingand} and the
references therein. One specific preprocessing step refers to detecting a
\emph{fixed variable}. These are generally detected when the
upper and lower bounds on a variable are equal. Fixed variables can
also be detected when an invertible block $A_{11}$ can be isolated
$A = \begin{bmatrix} A_{11} & A_{12}=0\cr A_{21} & A_{22}
\end{bmatrix},\, b = \begin{pmatrix} b_{1} \cr b_{2}\end{pmatrix}$.
With
$x = \begin{pmatrix} x_{1} \cr x_{2}\end{pmatrix}$, we can eliminate
$x_1= A_{11}^{-1}b_1$ and discard the first block of now redundant rows,
along with the first block of columns. 
If  $b_1=0$ then we have trivially identified variables fixed at zero
and removed redundant rows and columns. The remaining block $A_{22}$
remains full row rank as happens in Gaussian elimination.

In general, \FR for linear programs refers to identifying variables fixed
at $0$, and removing them along with corresponding columns and redundant
rows. In general, this is not as simple as above, and
the theorem of the alternative is needed.
As a consequence of our main result, we see below
that a single step of the simplex method, a phase-I part B approach,
yields many of these variables that are identically zero on the feasible
set.

One of the standard assumptions in linear programming is full row rank 
of $A$. As we observed in \Cref{lemma:ConstrRedundant}, 
each \FR step results in linear dependence of the constraints.
We now summarize two available methods for extracting a maximal 
linearly independent subset of rows of $AV$. 
The first method uses a rank-revealing QR decomposition\footnote{\url{https://www.mathworks.com/matlabcentral/fileexchange/77437}}. 
Let $M =(AV)^T$. 
Let $M I(:,\pi) =QR$ be a QR factorization
where $\pi$ is a permutation vector, $Q$ is a orthogonal matrix and $R$ is an upper triangular matrix with a non-increasing diagonal in absolute value. 
The matrix $I(:,\pi)$ permutes the columns of $M$.
If $M$ has linearly dependent columns, then the matrix $R$ contains zeros on its diagonal. 
Let $r$ be the number of the nonzero diagonal entries of $R$. 
Then, $\pi(1:r)$ returns the subset of columns indices of $M$ that are linearly independent. 
Another available method makes use of artificial variables \cite[Box 8.2]{MR717219}.
It constructs $\begin{bmatrix} I & AV \end{bmatrix}$ and sets the initial basis matrix to be the first $m$ columns.
Then it performs a variant of the phase-I of the two-phase simplex method to drive the basic variables out of the basis one by one. 
When such an operation is not applicable, a 
linearly dependent row of $AV$ is detected. 
Computational improvements of this method are made 
in~\cite{AndersenErlingD95,springerlink:10.1007/s00291-003-0130-x}.
A more recent method is the
rank revealing Gaussian elimination by the maximum volume concept
given in~\cite{MR4054120}.

\section{Main Result and Consequences}
\label{sec:MainResult}

In this section we present our main result, see~\Cref{thm:LPdegen}.
We provide two proofs:
one takes an algebraic approach by using the definition of the basic
feasible solution; and the other takes a geometric approach by 
using extreme points. 
Both proofs rely heavily on~\Cref{lemma:ConstrRedundant}.
In \Cref{sec:phase1partB} we present an efficient preprocessing scheme that can be used as an extension of the phase-I of the two-phase simplex method.
In~\Cref{sec:Discussions} we include immediate corollaries of the main result and interesting discussions.

\subsection{Lack of Strict Feasibility and Relations to Degeneracy}
\label{sec:LackSlaterANDDegen}

\begin{theorem}
\label{thm:LPdegen}
Suppose that strict feasibility fails for $\cF$ .
Then every basic feasible solution to $\cF$ is degenerate.
\end{theorem}

\subsubsection{An Algebraic Proof of~\Cref{thm:LPdegen} via the Definition of \BFS}
\label{sect:algproof}

\begin{proof}
Since there is no strictly feasible point in $\cF$, there exists a facial range vector $V$, and as in~\cref{eq:setEqauivalence} we have
\[
\cF = \{ \ x \  = Vv  \in \Rn  \ : \ AVv = b, \ v\in \R^{n-s_z}_+  \ \}.
\]
By \Cref{lemma:ConstrRedundant}, $AV$ has at least one redundant row.
By permuting the columns of $A$, we may assume that the matrix $V$ is of the form 
\[
V = \begin{bmatrix} I_{r}  \\ 0  \end{bmatrix}  \text{ and } r = n-s_z.
\]
We partition the index set $\{1,\ldots,n\}$ as
\[
\{1,\ldots,n\}= \cI_+ \cup \cI_0, \text{ where } \cI_+ = \{1, \ldots, r\}  \text{ and } \cI_0 = \{ r+1,\ldots ,n\} .
\]
Then we have $A = \begin{bmatrix} A_{\cI_+} & A_{\cI_0} \end{bmatrix}$.
Let $\bar{x} \in \cF$ be a  basic feasible solution with basic indices
\[
\cB \subset \{1,\ldots,n\}, \ |\cB|=m, \ \det( A_\cB) \neq 0, \ \text{
and } \ A_\cB \bar x (\cB) = b .
\]
Suppose $\cB\subseteq \cI_+$.
We note, by \Cref{lemma:ConstrRedundant} again, that $A_{\cI_+}=
AV$ has linearly dependent rows, i.e., $\rank(A_{\cI_+}) < m$. 
Hence $\bar{x}$ must include a basic variable in $\cI_0$ and this concludes that every basic feasible solution is degenerate.
\end{proof}

\subsubsection{A Geometric Proof Using Extreme Points}
\label{sec:geometricProof}

We now give the second proof of our main result.
Suppose that $X\in F$ with $\rank(X)=r$, where $F$ is a face of the set
$\{X \in \Snp: \trace(A_iX)=b_i, \forall i=1 ,\ldots, m\}$. Here, $\Snp$ denotes the set of $n$-by-$n$ positive semidefinite matrices. 
It is known that 
$ \frac{ r(r+1)}{2} \le m+\dim F$, see~\cite[Theorem 2.1]{GPat:95}.
We rewrite \cite[Theorem 2.1]{GPat:95}
in the language of polyhederon in \Cref{coro:PatakiLPversion}.   
We include the proof for completeness in \Cref{sec:appendixLPataki}.

\begin{cor}(\!\!\cite[Theorem 2.1]{GPat:95})
\label{coro:PatakiLPversion}
Suppose that $x\in F$, where $F$ is a face of the set $\cF$.
Let $r$ be the number of nonzeros in $x$ and $d = \dim F$.
Then the number of nonzero entries of $x \in F$ is at most $ m+d$.
\end{cor}

\index{face}
A point $x$ in a convex set $\cC$ is called an \textdef{extreme point} if, for all $y,z\in \cC$, $x = \frac{1}{2}(y+z)$ implies $x = y = z$.
An extreme point is itself a face and the dimension of this face is $0$.
Hence, we obtain \Cref{coro:PatakiLPextpt}
by writing \Cref{coro:PatakiLPversion} through the lens of extreme points.
\begin{cor}
\label{coro:PatakiLPextpt}
Every extreme point $x \in \cF$ has at most $m$ positive entries. 
\end{cor}

We now restate the main result of this paper~\Cref{thm:LPdegen}
in the language of extreme points and number of rows of $A$.
\begin{thm}
\label{thm:degeneracyPataki}
Suppose that strict feasibility of $\cF$ fails.
Then every extreme point $x \in \cF$ has at most $m-1$ positive entries.
\end{thm}

\begin{proof}
Since strict feasibility fails for $\cF$, we have $\cF = \{x=Vv \in \Rn : AVv = b, \ v\in \R_+^{n-s_z}\}$; see \cref{eq:setEqauivalence}.
From \Cref{lemma:ConstrRedundant}, we note that at least one equality in $AVv=b$ is redundant. 
Let $P_{\bar{m}} AV v = P_{\bar{m}} b$ be the system obtained after removing redundant rows of $AV$; see \cref{eq:setEqauivalencefullrow}.
Then, by \Cref{coro:PatakiLPextpt},
every extreme point of the set $\{v \in \R^{n-s_z}_+: P_{\bar{m}} AV v = P_{\bar{m}} b \}$  has at most $m-1$ nonzero entries. 
Hence, the statement follows.
\end{proof}

\index{$P_{\bar{m}}AV$}
\index{$P_{\bar{m}}b$}

\subsubsection{Immediate Consequences of Main Result}

%
We first note that
\Cref{thm:LPdegen} and \Cref{thm:degeneracyPataki} are equivalent owing
to the well-known characterization:
\[
x \in \cF \text{ is a basic feasible solution } 
\iff 
x \in \cF \text{ is an extreme point. } 
\]
We now highlight that \Cref{thm:LPdegen,thm:degeneracyPataki} do
not merely imply the existence of a \emph{single} degenerate basic
feasible solution; but rather that \emph{every} 
basic feasible solution is degenerate. 
Developing a pivot rule that prevents the simplex method from visiting
degenerate points is not possible as it can never avoid degeneracies
when strict feasibility fails, as we now illustrate in the following.
\begin{example}
\label{example:noIntDegen}
Consider $\cF$ with the data
\[
A = \begin{bmatrix}
1 & 1 & 3 & 5 & 2 \\ 0 & 1 & 2 & -2 & 2
\end{bmatrix} \text{ and  }
b = \begin{pmatrix}  1 \\ 1 \end{pmatrix} .
\]
Consider the vector 
$y = \begin{pmatrix} 1 \\ -1 \end{pmatrix}$. Then
\[
A^T y = \begin{pmatrix} 1 & 0 & 1 & 7 & 0  \end{pmatrix}^T \text{ and }
b^Ty = 0.
\]
Hence, \Cref{prop:Farkas} certifies that $\cF$ does not contain a strictly feasible point.
There are exactly six feasible bases in $\cF$.
The \BFS associated with $\cB \in \{\{1,2\},\{2,3\},\{2,4\}\}$ is
$ x = \begin{pmatrix} 0 & 1 & 0 & 0 & 0  \end{pmatrix}^T$; and
the \BFS associated with $\cB \in \left\{
\{1,5\},\{3,5\},\{4,5\}\right\}$ is
$ x = \begin{pmatrix} 0 & 0 & 0 & 0 & \frac{1}{2}  \end{pmatrix}^T$.
Clearly, all \BFSp s are degenerate. 
\end{example}

Recall that strict feasibility is equivalent to the
Mangasarian-Fromovitz constraint qualification,~\cite{MR0325154}. 
The latter is equivalent to
stability with respect to perturbations of $b$, and to a compact dual
optimal set.
Therefore, the following \Cref{coro:contraPosLPdegen}, obtained by writing the contrapositive 
of \Cref{thm:LPdegen}, is extremely interesting and important.
We provide \Cref{example:oneDegen} below to illustrate 
\Cref{coro:contraPosLPdegen}.
\begin{cor}
\label{coro:contraPosLPdegen}
Suppose that there exists a nondegenerate basic feasible solution. Then
there exists a strictly feasible point $\hat x \in \cF$.
\end{cor}

\begin{example}
\label{example:oneDegen}
Consider $\cF$ with the data
\[
A = \begin{bmatrix}
1 & 0 & -2 & 3 & -4 \\ 0 & -1 & -2 & 3 & 1
\end{bmatrix} \text{ and  }
b = \begin{pmatrix}  1 \\ 1 \end{pmatrix} .
\]
The system $\cF$ has exactly four feasible bases; the \BFS associated with $\cB \in \{ \{1,4\},\{2,4\},\{4,5\} \}$ is $x = \begin{pmatrix} 0&0&0&1/3&0\end{pmatrix}^T$
and the \BFS associated with $\cB = \{1,5\}$ is 
$x = \begin{pmatrix} 5&0&0&0&1\end{pmatrix}^T$.
We note that the \BFS associated with $\cB = \{1,5\}$ is nondegenerate.
As \Cref{coro:contraPosLPdegen} states, the system $\cF$ has a strictly feasible point, and it is verified by the point  
$\frac{1}{10} \begin{pmatrix} 4 & 1 & 1 & 4 & 1\end{pmatrix}^T$.
\end{example}
\Cref{coro:contraPosLPdegen} provides a useful check for strict
feasibility when the simplex method is used,
i.e.,~if there is any simplex iteration that yields a nondegenerate
\BFSp, then it is useful to record that occurrence.
We emphasize that recording the occurrence of a nondegenerate iteration
is inexpensive and the occurrence gives a certificate of the stability
of the \LP instance. 
We revisit \Cref{coro:contraPosLPdegen} in \Cref{sec:nondegentoSlater}
below and present an efficient algorithm for obtaining a Slater point
from a nongenerate \BFSp. But,
\Cref{example:converseCounterex}  below shows that the converse of 
\Cref{thm:LPdegen,thm:degeneracyPataki} is not true.
In other words, strict feasibility holds and every \BFS is degenerate.
\begin{example}
\label{example:converseCounterex}
\begin{enumerate}
\item
Consider $\cF$ with the data
\[
A = \begin{bmatrix}
1 & 0 & 2 & 0 & -2 \\ 1 & -3 & 2 & 1 & -2
\end{bmatrix} \text{ and  }
b = \begin{pmatrix}  1 \\ 1 \end{pmatrix} .
\]
$\cF$ has exactly four feasible bases and all of them are degenerate;
the \BFS associated with $\cB \in  \{\{1,2\},\{1,4\}\}$ is $x = \begin{pmatrix} 1 & 0 & 0 & 0 & 0 \end{pmatrix}^T$ and the 
\BFS associated with $\cB \in  \{\{2,3\},\{3,4\}\}$ is $x = \begin{pmatrix} 0 & 0 & 1/2 & 0 & 0 \end{pmatrix}^T$.
However, $\cF$ contains a strictly feasible point 
$\frac{1}{10} \begin{pmatrix} 1&1&5.5&3&1 \end{pmatrix}^T$.

\item 
Note that the linear assignment problem (marriage problem) has a
strictly feasible point but all the \BFS are highly degenerate\footnote{Note that this is true for the
transportation and the assignment problems. Both are highly degenerate
at each \BFS but satisfy strict feasibility. For example, for the
assignment problem order $n$, 
the feasible set can be considered to be the doubly
stochastic matrices $X$. The extreme points are the permutation matrices by
the Birkoff-Von Neumann theorem. Therefore, each extreme point has
exactly $n$ positive elements while there are $m=2n-1$ linearly
independent constraints.}.
Therefore, $\cI_0=\emptyset$; the set of variables fixed at $0$ is empty.
Moreover, as an \LPp, the problem is stable with respect to
perturbations in the data.
\end{enumerate}
\end{example}

From \Cref{example:noIntDegen,example:converseCounterex}, we observe 
that there are two different types of degeneracies. One involves
variables that are $0$ in one \BFS but positive in another; the second
involves variables fixed at $0$, i.e.,~that result in strict feasibility
failing. Note that strict feasibility (along with $A$ full row rank) 
is the Mangasarian-Fromovitz constraint qualification 
which is equivalent to stability with respect to 
right-hand side perturbations~\cite{MR1358394}, which is in turn
equivalent to a bounded dual optimal set.


Given a \BFS $\bar{x} \in \cF$, we let the 
\textdef{degree of degeneracy} of $\bar{x}$ denote the 
number of $0$'s among its basic variables.
By exploiting the facially reduced model we can check how degenerate the \BFSp s of $\cF$ are.

\index{$\cI_0$}

\begin{cor}
\label{cor:degrdegAidentz}
Suppose that strict feasibility fails for $\cF$, and let $\cF$ have the
facial range vector representation in~\cref{eq:setEqauivalence}.
Recall that the set of indices $\cI_0 = \{ i \in \{1,\ldots, n \} : x_i = 0, \ \forall x\in \cF \}$. 
Let $\bar{x}\in \cF$ be a basic feasible solution with basis $\cB$.
Then, the following holds.
\begin{enumerate}
\item \label{item:degDegen3}
The basis $\cB$ has an nonempty intersection with $\cI_0$, i.e.,
$\cB\cap \cI_0 \neq \emptyset$. 
\item 
\label{item:singletondeg}
If the degree of degeneracy of $\bar x$ is exactly one, with
$\bar x_k = 0, k\in \cB$,  then $x_k,A_{:,k}$ can be discarded from the
problem.
\item \label{item:degDegen1}
The degree of degeneracy of $\bar{x}$ is at least $m-\rank(AV)$.
\item \label{item:degDegen2}
At least $m-\rank(AV)$ number of basic indices of $\bar{x}$ are contained in $\cI_0$.
\end{enumerate}
\end{cor}

\begin{proof}
\begin{enumerate}
\item
Let $\bar{x} \in \cF$ be a basic feasible solution and let $\cB$ be a basis for $\bar{x}$. 
\Cref{item:degDegen3} follows from the proof and the definition of the set $\cI_0$ of elements $x_i$ that are identically zero on the feasible set.
\item
The proof follows from the algebraic proof of \Cref{thm:LPdegen} given
in~\Cref{sect:algproof}.
Since every \BFS is degenerate and the basis has a nonempty intersection with $\cI_0$, the index $k$  must be in $\cI_0$.

\item
For \Cref{item:degDegen1}, we note that $A_\cB $ contains linearly independent columns. 
Then $A_\cB $ can contain at most $\rank(AV)$  number of columns from $AV$.
Thus, $\bar{x}(\cB)$ must contain at least $m-\rank(AV)$ number of zeros.
\item
\Cref{item:degDegen2} is a direct consequence of \Cref{item:degDegen3} and \Cref{item:degDegen1}.
\end{enumerate}
\end{proof}
\Cref{item:degDegen1,item:degDegen2} of \Cref{cor:degrdegAidentz} are closely related to the implicit problem singularity, $\ips$, and the max-singularity degree, $\maxsd$; see \Cref{def:ips}.
In particular, $\ips(\cF)$ is a lower bound of the degree of degeneracy of every \BFS of $\cF$; the more implicit redundancies $\cF$ contains, the more degenerate every \BFS becomes.
We include an alternative way to view \Cref{cor:degrdegAidentz}
in \Cref{sec:geometricProof}.

\index{$\ips$}
\index{$\maxsd$}


We conclude the discussions with the following interesting observation.
This again illustrates the implicit singularity of the constraints when
the Slater condition fails.
\begin{cor}
Suppose that strict feasibility fails for $\cF$ and that $m=1$. Then the
trivial $x^*=0$ is an optimal solution.
\end{cor}

\subsection{Efficient Preprocessing for Facial Reduction and Strict Feasibility}
\label{sec:phase1partB}

In this section we present an efficient preprocessing method for obtaining a facially reduced system. 
In \Cref{sec:nondegentoSlater} we discuss obtaining a strictly feasible 
point using a nondegeneate \BFS and its variant.
In \Cref{sec:partBofPhaseI} we consider the general case of finding an
exposing vector to obtain the facially reduced strictly feasible \LPp.

\subsubsection{Towards a Strictly Feasible Point from a Nondegenerate \BFS}
\label{sec:nondegentoSlater}

By \Cref{coro:contraPosLPdegen}, 
the existence\footnote{Determining the existence of a degenerate basic feasible solution is an NP-complete problem; see \cite{MR668279}.} of a nondegenerate \BFS guarantees the existence of a strictly feasible point.  
We now propose a process for acquiring a Slater point from a
nondegenerate \BFSp, and include a generalization.
The arguments in this section also provide a constructive proof of \Cref{coro:contraPosLPdegen}.

Let $\bar{x}\in \cF$ be a nondegenerate \BFSp.
Without loss of generality, we assume that the (all positive) basic variables $\bar{x}_\cB$ of $\bar{x}$ are located at the last  $m$ entries of $\bar{x}$.
We fix a scalar $\hat{\gamma} \in (0,1)$ and an index $j\in \{1,\ldots,
n-m\}$. For some $\alpha \ge 0$, we consider the 
simplex method \emph{ratio test} type inequality
\begin{equation}
\label{eq:basicGivingAway}
\hat{\gamma} \bar{x}_\cB  - \alpha( A_\cB )^{-1} A_j \ge 0.
\end{equation}
Since $\bar{x}_\cB >0,\hat \gamma > 0$, there exists a \emph{positive} 
$\alpha$ that maintains the inequality \cref{eq:basicGivingAway}.
Let
\begin{equation}
\label{eq:steplengthSlater}
\alpha^* 
= \min\left\{ 1, \ \max\{ \alpha \in \R_+ : \hat{\gamma} \bar{x}_\cB  - \alpha( A_\cB )^{-1} A_j \ge 0 \} \right\},
\end{equation}
and decompose 
\[
\hat{\gamma} \bar{x}_\cB = 
\left( \hat{\gamma} \bar{x}_\cB  - \alpha^*( A_\cB )^{-1} A_j \right)
+ \alpha^*( A_\cB )^{-1} A_j.
\]
We observe that
\[
\begin{array}{rcl}
b &=& A_{\cB} \bar{x}_\cB \\
&=&(1-\hat{\gamma} ) A_{\cB} \bar{x}_\cB  + \hat{\gamma} A_{\cB} \bar{x}_\cB \\
&=& (1-\hat{\gamma} ) A_{\cB} \bar{x}_\cB  
+ A_\cB \left( \hat{\gamma} \bar{x}_\cB  - \alpha^*( A_\cB )^{-1} A_j
+ \alpha^*( A_\cB )^{-1} A_j \right) \\
&=& A_\cB (\bar{x}_\cB - \alpha^*( A_\cB )^{-1} A_j )  + \alpha^* A_j .
\end{array}
\]
If we set $x_j = \alpha^* >0$ and replace $\bar{x}_\cB$ by $\bar{x}_\cB - \alpha^*( A_\cB )^{-1} A_j $, then we have increased the cardinality of the positive entries of a solution. 
We note that $\bar{x}_\cB - \alpha^*( A_\cB )^{-1} A_j $ only has
strictly positive entries since it it a sum of a positive vector and a nonnegative vector;
\[
\bar{x}_\cB - \alpha^*( A_\cB )^{-1} A_j 
= \underbrace{(1-\hat{\gamma} ) \bar{x}_\cB }_{\text{positive}}+ 
\underbrace{ \hat{\gamma} \bar{x}_\cB  - \alpha^*( A_\cB )^{-1} A_j }_{\text{nonnegative}}.
\]

We can continue to increase the number of positive entries of a solution 
one by one for each $j\in \{1,\ldots,n-m\}$.
Moreover, we can achieve this by a compact vectorized operation.
The main idea is that we can choose $\hat{\gamma}$ in \cref{eq:basicGivingAway} independently for each $j \in \{1,\ldots,n-m\}$.
Let $\gamma_j$ be a positive real number such that
$0< \gamma := \sum_{j=1}^{n-m} \gamma_j <1$. Then, we have
\[
\bar{x}_\cB = 
(1-\gamma) \bar{x}_\cB  + \gamma \bar{x}_\cB =
(1-\gamma)  \bar{x}_{\cB}+ \sum_{j=1}^{n-m} \gamma_{j} \bar{x}_\cB .
\]
We set an auxiliary matrix 
\[
\Theta =  \begin{bmatrix} \gamma_1 \bar{x}_\cB & \cdots & \gamma_{n-m} \bar{x}_\cB \end{bmatrix}
 -(A_\cB)^{-1} A_{1:n-m}   \in \R^{m \times (n-m)}
\]
and perform \cref{eq:steplengthSlater} on each column $j$ of $\Theta$ to obtain the vector $\theta^*$:
\[
\theta^*_j := 
\begin{cases}
\max(\Theta(:,j)) & \text{ if } \max(\Theta(:,j))\le 1, \\
1 & \text{ otherwise}.
\end{cases}
\]
Then the point 
\[
\begin{bmatrix}
\theta^* \\ \bar{x}_\cB - (A_\cB)^{-1} A_{1:n-m} \theta^*
\end{bmatrix} 
\]  
is a strictly feasible point to $\cF$. 
Hence, this operation provides a constructive proof of \Cref{coro:contraPosLPdegen}.

We now extend the aforementioned procedure for obtaining a strictly 
feasible point using any feasible solution $\bar{x}\in \cF$ such that
$A_{\supp(\bar{x})}$ is full row rank.  We partition $\bar{x}\in 
\cF$ as follows
\begin{equation}
\label{eq:partitionxbar}
\bar x=\begin{pmatrix} \bar x_{\cB_1} \\
\bar x_{\cB_2} \\ \bar x_\cN \end{pmatrix},\,\text{ where }
\supp(\bar{x}) =  \cB_1 \cup \cB_2, \ \rank(A_{\cB_1)} =m , \ \text{ and } 
\cN = \{1,\ldots,n\} \setminus \supp(\bar{x}).
\end{equation}
We partition $A$ using the same partition $\cB_1\cup \cB_2 \cup \cN$:
\[
\begin{bmatrix} A_{\cB_1} & A_{\cB_2} &  A_\cN \end{bmatrix} \bar{x} = b 
\iff
\begin{bmatrix} A_{\cB_1} &  A_\cN \end{bmatrix}
\begin{pmatrix} \bar{x}_{\cB_1} \\ \bar{x}_\cN \end{pmatrix} = \bar{b} := b -  A_{\cB_2}  x_{\cB_2}  .
\]
Then we can apply the aforementioned procedure to the system 
\[
\begin{bmatrix} A_{\cB_1} &  A_\cN \end{bmatrix}
\begin{pmatrix} \bar{x}_{\cB_1} \\ \bar{x}_\cN \end{pmatrix} = \bar{b} 
\]
and distribute positive weights to $\bar{x}_\cN$ using $\bar{x}_{\cB_1}$.
Finally, we find a strictly feasible point to $\cF$.
This process is summarized in  \Cref{algo:Nondegen2Slater}.
Furthermore, \Cref{algo:Nondegen2Slater} provides a constructive proof for 
\Cref{prop:genenralBFStoSlater} below.
\begin{prop}
\label{prop:genenralBFStoSlater}
Let $x\in \cF$ be a solution such that $\rank \left( A_{\supp(x)}) \right)=m$.
Then, $\cF$ has a strictly feasible point. 
\end{prop}

\begin{algorithm}
\begin{algorithmic}[1]
\REQUIRE Given: $A, \ \bar{x}\in \cF$ partitioned as in~\cref{eq:partitionxbar}.
\STATE Choose any $\gamma \in \R^{|\cN|}_{++}$ such that $\sum_{j=1}^{|\cN|} \gamma_j <1$.
\STATE Compute 
\[
\Theta = \begin{bmatrix}
\bar{x}_{\cB_1} & \cdots & \bar{x}_{\cB_1}\end{bmatrix}  \Diag(\gamma) - A_{\cB_1}^{-1} A_\cN .
\]
\STATE Compute $\theta^* \in \R_{++}^{|\cN|} $, where for each $j \in \{1,\ldots,|\cN|\}$, 
\[
\theta^*_j := 
\begin{cases}
\max(\Theta(:,j)) & \text{ if } \max(\Theta(:,j))\le 1, \\
1 & \text{ otherwise}.
\end{cases}
\]
\STATE 
Set $x^\circ = \begin{pmatrix}
\bar{x}_{\cB_1} - (A_{\cB_1})^{-1} A_\cN \theta^* \\ \bar{x}_{\cB_2} \\  \theta^*
\end{pmatrix}$.
\end{algorithmic}
\caption{Compute a Slater Point}
\label{algo:Nondegen2Slater}
\end{algorithm}

\subsubsection{Exposing Vector; Phase I Part B; Strict Feasibility Testing}
\label{sec:partBofPhaseI}

We now present an efficient preprocessing procedure for detecting 
identically $0$ variables and obtaining exposing vectors in order
to get the facially reduced \LPp. We do this for
a given \BFS $\bar{x}$ by solving special subproblems using the 
simplex method. By the end of the process, we determine one of:
\begin{enumerate}
\item  a certificate $y$ that produces an exposing vector $A^Ty$ (Slater condition fails);
\item  a strictly feasible point (Slater condition holds).
\end{enumerate}

This process in fact has \emph{two} applications. 
First, since the only requirement of this process is the \BFSp, 
the procedure can be considered as an extension of phase-I of the 
two-phase simplex method that obtains the equivalent facially reduced problem.
Second, the procedure can be used as a postprocessing step. 
We could perform \FR on the optimal face and find, and delete, variables
fixed at zero in order to improve stability of the optimal solution.

\index{$\cI_0$}
\index{$\cB_0$}

We now describe the proposed preprocessing method.
Let $\cB$ be a degenerate initial basis of $\cF$ with associated \BFS
$\bar{x}$.
Without loss of generality, we assume that basic variables are located at the first $m$ entries of $\bar{x}$.
Let $d$ be the degree of degeneracy of $\bar{x}$.
We further assume that the degenerate basic variables are located at the first $d$ entries of $\bar{x}$.
We let $\cB_0 := \{1,\ldots, d\}$. We now test and record whether or not
each $i\in\cB_0$ is a variable fixed at $0$.
Let $i\in \cB_0$, and consider the following problem:
\begin{equation}
\label{eq:max_X1}
p^*_i = \max \{x_i \,:\, Ax=b, \, x\geq 0\}.
\end{equation}
We may assume that $i=1$.
We solve \cref{eq:max_X1} using the simplex method from the initial
\BFS $\bar{x}$. That is, we do not need to perform the typical phase-I of the two-phase simplex method in order to find a feasible \BFSp.
The optimal value $p_1^*$ of \cref{eq:max_X1} is clearly lower bounded by $0$.
We consider two cases below:
\begin{enumerate}
\item Suppose that $x_1>0$ after $k$ iterations.
Then, the variable $x_1$ is not an identically $0$ variable, 
i.e.,~we record that $1\in \cI_+$.
\item Suppose that $p_1^*=0$. 
Then, the variable $x_1$ is an identically $0$ variable, i.e.,~we record
that $1\in \cI_0$.
Let $\cB^*$ be an optimal basis for \cref{eq:max_X1}. Then we have
\begin{equation}
\label{eq:maxXidualfeasible}
y^* = A_{\cB^*}^{-T} e_1, \  \<b,y^*\>=0 \ \text{ and }
A^Ty^*  \ge e_1 ,
\end{equation}
where $e_1$ is the first unit vector of appropriate dimension.
We note that the dual optimal solution $y^*$ in
\cref{eq:maxXidualfeasible} produces a solution to the auxiliary
system~\cref{eq:auxsystem}.
Therefore, we obtain a \emph{nontrivial} exposing vector since $0\neq A^Ty^*\geq 0$.
\end{enumerate}

Let $\{y^j\}$ be a collection of the certificates that are obtained
from solving \cref{eq:max_X1} with the index~$1$ replaced by $i\in \cB_0$. Then $y^\circ = \sum_j y^j$ is also a
certificate, i.e.,
\[
A^T y^\circ = \sum_j A^T y^j \ge 0,\ 
A^Ty^\circ \ne 0, \ \text{ and } \
\<b,y^\circ\> = \sum_j \<b,y^j\> = 0,
\]
and we obtain a nontrivial exposing vector $A^T y^\circ$ for the system $\cF$.
By summarizing the two cases above, we obtain an efficient preprocessing method \Cref{algo:phase1partB}.
\begin{algorithm}
\begin{algorithmic}[1]
\REQUIRE A \BFS $\bar{x}$ with corresponding basis $\cB$; set 
$\cB_0 = \{ i\in \cB : \bar{x}_i =0 \}$
\STATE \textbf{Initialize:}
$x^\circ = \bar{x}$, $y^\circ = 0 \in \Rm$, $\cJ_0 = \emptyset$,
$\cB_*\leftarrow \cB_0$
\WHILE{ $\cB_0 \ne \emptyset$ and $\cB_* \ne \emptyset$}
\STATE  Pick $i \in \cB_0$; starting from the initial \BFS $\bar{x}$, solve 
for primal-dual optima $x^*,y^*$
\[
x^* = \argmax_x \{ x_i : Ax =b, x \ge 0 \}, \  \  p^* = x_i^* = b^Ty^*
\]
But, if during the solve, $x_i>0$, then stop the iterations; set
$x^*$ as the current point.
\STATE $\cS \leftarrow \supp(x^*)$
\STATE $\cB_* \leftarrow$ degenerate basic indices for $x^*$
\IF{ $\cB_0 \neq \emptyset$ and  $\cB_* \neq \emptyset$ }
\IF{$p^*=0$ (strict feasibility fails)}
\STATE Use dual certificate $y^*$ to  satisfy \cref{eq:auxsystem}
\STATE $y^\circ \leftarrow y^\circ + y^*$
\STATE $\cJ_0 \leftarrow \cJ_0 \cup (\supp(A^Ty^*) \cap \cB)$
\STATE $\cB_0 \leftarrow \cB_0 \setminus \{ \cS \cup \cJ_0 \}$
\ELSE 
\STATE $\cB_0 \leftarrow \cB_0 \setminus \cS$
\ENDIF
\STATE Choose $\gamma \in (0,1)$ and set $x^\circ \leftarrow \gamma x^\circ + (1-\gamma) x^*$
\ENDIF
\ENDWHILE 
\IF{$\cJ_0 \ne \emptyset$}
\STATE $z^\circ = A^T y^\circ$ (exposing vector)
\STATE  $\cR \leftarrow$  redundant row indices of $A\left(:,  \supp(z^\circ)^c \right)$
\STATE  $A\leftarrow A(\cR^c,\supp(z^\circ)^c), \  b \leftarrow b(\cR^c)$
\ELSE
\STATE Run \Cref{algo:Nondegen2Slater} with $x^\circ$ and $\det(A_\cB)\ne 0$ (use $x^*$ and $\cB_*$, if $\cB_* = \emptyset)$   
\ENDIF
\end{algorithmic}
\caption{Preprocessing Phase I Part B; Towards Strict Feasibility}
\label{algo:phase1partB}
\end{algorithm}

The following allows for simplifications in~\Cref{algo:phase1partB}.
\begin{lemma}
\label{lemma:1AlwaysinBasis} 
Let $\cB$ be an initial basis containing the index $i$ for problem \eqref{eq:max_X1}. Then the index $i$ always remains in the basis throughout the iterations. 
\end{lemma}
\begin{proof}
Without loss of generality, we let $i=1$. We argue that $1$ is not chosen to leave the basis. Let $y^* = (A_\cB^T)^{-1} c_\cB$ and $\bar{A} = A_{\cB}^{-1} A$.
Suppose that the reduced cost at the index $j$ is positive. Then 
\[
\begin{array}{rcl}
0<\bar c_j 
= c_j-A_j^Ty^*
= -A_j^Ty^*
= -A_j^T(A_\cB^T)^{-1} e_1
= -\bar A_{1j}.
\end{array}
\]
Since $\bar A_{1j} <0$,  the index $1$ is not chosen to leave the basis $\cB$.
\end{proof}

The following special case is of interest.
Namely, no simplex pivoting steps are required to determine strict feasibility. 
\begin{theorem}(preprocessing for degree of degeneracy $1$)
\label{thm:effFR}
Given a basis $\cB$, let $\bar{x}$ be a \BFS with the
degree of degeneracy exactly one and with $\bar x_i=0, i\in \cB$. 
Let $\cN = \{1,\ldots,n\} \setminus \cB$ and let
 $\bar y = (A_\cB^T)^{-1} c_\cB, c_\cB = e_i$. 
Then strict feasibility fails if, and only if,~$\bar y$
satisfies~$A^T_\cN \bar y \ge 0$.
\end{theorem}

\begin{proof}
Suppose that $\bar x$ is a degenerate \BFS with basis $\cB$.
Without loss of generality, we assume $1\in\cB$ and $1$ is the degenerate index. We consider the problem
\[
p^*_1 = \max \{x_1 \,:\, Ax=b, \, x\geq 0\}.
\]
We note that $\<b,\bar y\> = 0$ since $\<b,\bar y\>$ is identical to the current objective value `$0$'.
The backward direction is clear by \Cref{prop:Farkas}.
Now suppose that strict feasibility fails. 
Suppose to the contrary that $A^T_\cN \bar y \ge 0$ fails.
Then there exists $j$ such that $A^T_j \bar y <0, \ j \in \cN$.
Note that, by \Cref{lemma:1AlwaysinBasis}, that $1$ is not chosen to leave the basis. 
Thus, there is an index $k \ne 1, k\in \cB$ that leaves the basis.
Since all other basic variables are positive, we obtain a positive step length and we improve the objective value, which yields a contradiction to $p^*_1=0$.
\end{proof}

Upon the termination of \Cref{algo:phase1partB}, we can always determine
whether the system $\cF$ has a strictly feasible point or not. 
\Cref{algo:phase1partB} terminates in a finite number of iterations since we remove at least one element from the set $\cB_0$ in each iteration. 
We emphasize that we do not need to solve the auxiliary \LPp s for all
$i\in \{1,\ldots,n\}$.  We solve \cref{eq:max_X1} only for the degenerate basic indices of the predetermined basis $\cB$.
However, upon termination of \Cref{algo:phase1partB}, it is possible
that we have not 
obtained $\face(\cF, \Rnp)$, the minimal face containing $\cF$. Although the
complete \FR for \LP can be completed in one iteration, one step
termination is possible only when we find a solution $y$ of 
\cref{eq:auxsystem} so that $A^Ty$ is in the relative interior of the
conjugate face of $\face(\cF,\Rnp)$. In this case, we can rerun
\Cref{algo:phase1partB} with the current facially reduced system. For
finding an initial basis for the second trial, we may use the efficient
basis recovery scheme~\cite[Chapter 7]{SWright:96}.

One of the nice features of \Cref{algo:phase1partB} is that we do not need to search for a new initial basis $\cB$ for each iteration; the initial basis remains the same. Therefore, our approach can be directly employed after the standard phase-I of the two phase simplex method.

We do not need a lot of pivoting steps to determine if $p_i^*$ is zero or positive. 
If $p_i^*=0$, the initial $\cB$ is indeed a basis that gives the optimal value. However the dual feasibility may not be obtained immediately\footnote{If we have a nondegenerate initial basis, then the dual feasibility is immediately obtained. However, our initial basis is degenerate.}. Thus, there may be additional pivots required to obtain the dual feasibility. However, since the optimal value is obtained at $\cB$, we do not expect that the optimal basis search to be time-consuming. 
For the case $p_i^* \in (0,\infty)$, the optimal value $p_i^*$ does not need to be
found. Hence once a basis that gives a positive support on $i$ is found,
we can terminate the maximization problem in \Cref{algo:phase1partB}  immediately. We recall from
\Cref{lemma:1AlwaysinBasis} that the index $i$ in \cref{eq:max_X1} never leaves the basis.
In the case of $p_i^*=\infty$, we can perform the following operation.
Let $\cB_c$ be a basis that indicates $p_i^*=\infty$ and let $j$ be an entering variable that indicates the unboundedness. 
Then by setting
\[
x^\circ(j) \leftarrow 1, \ 
x^\circ(\cB_c) \leftarrow x_{\cB_c} - A_{{\cB_c}}^{-1}A_j \text{ and } 
x^\circ((\{j\} \cup \cB_c)^c) = 0,
\]
we obtain a feasible solution $x^\circ$ that yields a positive objective value.

We often get an exposing vector that reveals more than one element in
the set $\cI_0$ by solving~\cref{eq:max_X1}. 
Let  $p_1^*=0$ in \cref{eq:max_X1} and let $y^*$ be a dual feasible solution. 
Suppose that $A^T y^* = e_1$, i.e., only one exposed variable is
revealed. Then $y^* \in \nul (A(:,2:n)^T )$.  
Since the data matrix $A$ has more columns than rows, $y^* \in \nul
(A(:,2:n)^T )$ generally implies $y^*=0$; this makes $A^Ty^* = e_1$
impossible.

\index{degree of degeneracy}

When an instance is large and have a \BFS with a very large degree of degeneracy, one may adopt parallel computing for \Cref{algo:phase1partB}
in order to reduce the total computation time. We note again that the
initial basis remains the same throughout the iterations. Hence, solving
\cref{eq:max_X1} for individual $i\in \cB_0$ can be performed independently.
In fact, parallel computing can be used to obtain a strictly feasible solution in \Cref{algo:Nondegen2Slater} as well; the weight vector $\gamma$ can be chosen independently for each $j \in \cN$.

\subsection{Discussions}
\label{sec:Discussions}

In this section we discuss the main result in 
\Cref{sec:LackSlaterANDDegen,sec:phase1partB} and
make connections to new results and known results in the literature.

\subsubsection{Distance to Infeasibility}
\label{sect:distinfeas}


\index{$\dist(b,\cF=\emptyset)$, distance to infeasibility}
\index{distance to infeasibility, $\dist(b,\cF=\emptyset)$}


The  \textdef{distance to infeasibility} is a measure of the
smallest perturbations of the data  $(A,b)$ of a problem that renders
the problem infeasible. In our setting, we can use the following simplification of the distance to infeasibility from \cite{Ren:94} by restricting the perturbation to $b$, i.e., we can force infeasibility using only perturbation in $b$;
\[
\dist(b,\cF=\emptyset) 
:= 
\inf \left\{ \ \|b-\tilde{b} \| \ : \ \{x\in \Rn : Ax=\tilde{b} , \  x\ge 0 \} =\emptyset \ \right\}.
\]
Many interesting bounds, condition numbers, are shown in \cite{Ren:94} 
under the assumption that the distance to infeasibility is positive and
known. 
It is known that a positive distance to infeasibility of $\cF$ implies that strict feasibility holds for $\cF$; see e.g.,~\cite{FrVe:97,FreundRobertM2005OaEo}.
The contrapositive of this statement is that, if strict feasibility fails for $\cF$, then the distance to infeasibility is $0$.
We revisit this statement with the facially reduced system
\cref{eq:setEqauivalence}. We provide an elementary proof that there is an arbitrarily small perturbation for the data vector $b$ of $\cF$ that yields the set $\cF$ infeasible, i.e., $\dist(b,\cF=\emptyset)=0$.
Furthermore, we provide explicit perturbations that render the set $\cF$ empty.

Suppose that $\cF$ fails strict feasibility. Recall the representation \cref{eq:setEqauivalence} for $\cF$.
Let $AV=QR$ be a QR decomposition of $AV$, where $Q\in \R^{m\times m}$ orthogonal,  $R\in \R^{m\times (n-s_z)}$ upper triangular.
We write $Q = \begin{bmatrix} Q_1 & Q_2 \end{bmatrix}$ so that $\range(Q_1) = \range(AV)$.
Then, by the orthogonality of $Q$, we have
\begin{equation*}
\label{eq:setHittingwithQ}
Ax = AVv = b
\iff
Q^TAx =  R v = Q^Tb .
\end{equation*}
Since $AV$ is a rank deficient matrix  (see \Cref{lemma:ConstrRedundant}), the upper triangular matrix $R$ is of the form
\begin{equation}
\label{eq:RofQRdecomp}
R = \begin{bmatrix}
\bar{R} \\ 0 \end{bmatrix}  \in \R^{m\times (n-s_z)} \text{ and } \bar{R}\in \R^{\rank(AV) \times (n-s_z)} \text{ with nonzero diagonal. }
\end{equation}
Since $b\in \range(AV)$, the last $m-\rank(AV)$ entries of $Q^Tb$ are equal to $0$, i.e.,
\[
Q^Tb 
= \begin{pmatrix} Q_1^Tb \\ Q_2^T b \end{pmatrix} 
= \begin{pmatrix} Q_1^Tb \\ 0 \end{pmatrix} .
\]
Consequently, the unrealized implicit non-surjuectivity produces the system 
\begin{equation}
\label{eq:reducedsysNoRemoval}
\begin{bmatrix} \bar{R} \\ 0 \end{bmatrix} v =
\begin{pmatrix} Q_1^Tb \\ 0 \end{pmatrix} , \ v \in \R_+^{n-s_z}.
\end{equation}
Any perturbation on the last $m-\rank(AV)$ equations in \cref{eq:reducedsysNoRemoval} that causes the system inconsistency
renders the system \cref{eq:reducedsysNoRemoval} infeasible while maintaining the dimension of $\relint(\cF)$.
For instance, replacing the right-hand side vector in~\cref{eq:reducedsysNoRemoval} by~$\begin{pmatrix} Q_1^Tb \\ \xi \end{pmatrix}$ with any nonzero vector $\xi \in \R^{m-\rank(AV)}$ renders \cref{eq:reducedsysNoRemoval} infeasible. 
Replacing the data matrix in \cref{eq:reducedsysNoRemoval} by $\begin{bmatrix} \bar{R} \\ \Phi \end{bmatrix}$ for which $\Phi$ contains a positive row vector also renders \cref{eq:reducedsysNoRemoval} infeasible.




We now present a class of perturbations of $b$ that maintains the feasibility of the set $\cF$ as well as a special perturbation of $b$ that forces $\cF$ to be infeasible. 
Such perturbations can be found using linear combinations of the columns of $Q_1$ or $Q_2$, respectively.
We relate this observation to the solution of the auxiliary system~\cref{eq:auxsystem} in the proof of~\Cref{prop:distInfrhsPert} below.

\begin{prop}
\label{prop:distInfrhsPert}
Suppose that strict feasibility fails for $\cF$, and let $\cF$ have the representation~\cref{eq:setEqauivalence}. Then the following hold.
\begin{enumerate}
\item \label{item:posDistInf}
For all $\Delta b \in \range(AV)$ with sufficiently small norm, 
the set $\{x\in \Rnp : Ax = b+ \Delta b\} $ is feasible.
\item  \label{item:zeroDistInf} 
Let $\bar{y} \in \Rm$ be a solution to the auxiliary system \cref{eq:auxsystem}. 
Then perturbing the right-hand side vector $b$ of $\cF$ in the direction
$\bar{y}$ makes the system $\cF$ infeasible. 
\end{enumerate}
\end{prop}
\begin{proof}
Let $\Delta b$ be any perturbation in $\range(AV)$. 
Let $QR=AV$ be a QR decomposition of $AV$. In particular, let $R$ have the form \cref{eq:RofQRdecomp} and $Q = \begin{bmatrix} Q_1 & Q_2 \end{bmatrix}$ so that $\range(Q_1) = \range(AV)$.
Let $\epsilon$ be a sufficiently small scalar. 
Then 
\begin{equation}
\label{eq:posdistQR}
Ax = AVv = b + \epsilon \Delta b 
\iff
Rv = Q^Tb + \epsilon Q^T \Delta b
\iff \bar{R} v = Q_1^T b + \epsilon Q_1^T \Delta b .
\end{equation}
The last equivalence holds since $Ax=b$ and $\Delta b \in \range(AV) = \range(Q_1)$.
Since the system $\bar{R} v = Q_1^T b $ satisfies the Mangasarian-Fromovitz constraint qualification, the distance to infeasibility of this system is positive.
Thus, the perturbed system $\{v: \bar{R} v = Q_1^T b + \epsilon Q_1^T \Delta b, \  v\ge 0\}$ remains feasible.
Therefore, by \cref{eq:posdistQR}, perturbing $\cF$ along the direction $\Delta b \in \range(AV)$ maintains the feasibility and this concludes the proof for \Cref{item:posDistInf}.

For \Cref{item:zeroDistInf} we show that perturbing $b$ with $\Delta b=\bar{y}$ renders $\cF$ infeasible, where $\bar{y}$ is a solution to the system \cref{eq:auxsystem}.
By \Cref{prop:Farkas} and \cref{eq:redunProof}, the nonzero vector $\bar{y}\in \Rm$ is in $\nul( (AV)^T)$.
Then we have
\[
\bar{y} \in \range (AV)^\perp = \range(Q_2)
\implies 
\bar{y} = Q_2 \bar{u}  \text{ for some nonzero } \bar{u} .
\]
We recall Farkas' lemma:
\[
\{y\in \Rm: A^Ty \ge 0, \ \<b,y\> <0 \} \ne \emptyset
\implies \cF = \emptyset.
\]
Now, for any $\epsilon >0$, setting $\Delta b_\epsilon = - \epsilon \bar{y}$ yields
\begin{equation}
\label{eq:FarkasInfea}
A^T \bar{y} \ge 0 ,\ \<b,\bar{y} \> = 0 
\implies 
A^T \bar{y} \ge 0 ,\ \<b+\Delta b_\epsilon ,\bar{y} \> < 0 .
\end{equation}
Hence, by letting $\epsilon \rightarrow 0^+$, we see that the distance to infeasibility, $\dist(b,\cF=\emptyset)$, is equal to $0$.
\end{proof}

We emphasize that the result 
\[
\cF \text{ fails strictly feasibility } \implies
\dist ((A,b),\cF =\emptyset) =0
\]
gives rise to the second step \cref{eq:setEqauivalencefullrow} of \FR discussed in \Cref{sec:FacialReduction}.
We note that the instability discussed in this section essentially
originates from the observation made in \Cref{lemma:ConstrRedundant},
i.e.,~redundant equalities arise in the facially reduced system.
Facially reduced system allows us to exploit the root of potential instability when the problem data $A$ or $b$ is perturbed.
Although the distance to infeasibility is $0$ in the absence of strict feasibility, 
\Cref{prop:distInfrhsPert} suggests that a carefully chosen perturbation of $b$ does not have an impact on the feasibility of $\cF$.
We provide a related numerical experiment in \Cref{sec:numericsDistInfes} below.




\subsubsection{Applications to Known Characterizations for Strict Feasibility}

There are some known characterizations for strict feasibility of $\cF$.
Using these characterizations we can obtain extensions of
\Cref{thm:LPdegen,thm:degeneracyPataki,coro:contraPosLPdegen}.

\index{($\cD$), dual of ($\cP$)}
\index{dual of ($\cP$), ($\cD$)}

The dual ($\cD$) of ($\cP$) is 
\begin{equation}
\label{eq:dualLP}
(\cD) \qquad \max_{y,s} \left\{ b^Ty \ : \ A^Ty + s = c, \ s\ge 0 \right\} .
\end{equation} 
It is known that strict feasibility fails for $\cF$ if, and only if, the set of optimal solutions for the dual $(\cD)$ is unbounded; see e.g., \cite[Theorem 2.3]{SWright:96} and \cite{GauvinJacques1977Anas}. Then \Cref{coro:dualSlaterChar} follows.
\begin{cor}
\label{coro:dualSlaterChar}
\begin{enumerate}
\item
Suppose that the set of optimal solutions for the dual $(\cD)$ is unbounded. 
Then every basic feasible solution to $\cF$ is degenerate.
\item
Suppose that there exists a nondegenerate basic feasible solution to $\cF$.
Then the set of optimal solutions for the dual $(\cD)$ is bounded.
\end{enumerate}
\end{cor}
It is known that strict feasibility holds for $\cF$ if, and only if, $b \in \relint(A(\Rnp))$, where $\relint$ denotes the relative interior; see e.g., \cite[Proposition 4.4.1]{DrusWolk:16}. 
Then if one finds a set of indices $\cI \subset \{1,\ldots,n\}$ such that $A_\cI$ is nonsingular and $A_\cI z = b$ has a solution $z$ with positive entries, then $b \in \relint(A(\Rnp))$.

\index{relative interior, $\relint$}
\index{$\relint$, relative interior}



\subsubsection{Applications to Obtain a Strictly Complementary Primal-Dual Solution}

In this section we present an application of \Cref{algo:Nondegen2Slater}
for obtaining a strictly complementary primal-dual optimal solution.

Let $(x^*,y^*,s^*)$ be an optimal triple for the standard primal-dual \LP pair. 
Let $\cB^* \cup \cN^* = \{1,\ldots,n\}$  be the strict complementary
partition of the primal-dual optimal pair.
The existence of such a partition is guaranteed by the Goldman-Tucker
theorem \cite{MR21:633} and the partition $\cB^* \cup \cN^*$ is unique.
For the first application of \Cref{algo:Nondegen2Slater}, we provide a method for obtaining a strict complementary primal-dual solution when the primal optimal solution $x^*$ is nondegenerate or the submatrix $A(:,\supp(x^*))$ of $A$ has rank $m$.
To elaborate, we list the two cases where \Cref{algo:Nondegen2Slater} can be used to obtain maximal complementary solutions.
\begin{enumerate}
\item
Let $x^*$ be a nondegenerate (optimal) basic feasible solution.
Then, $\supp(s^*)=\cN^*$ and $\supp(x^*)$ can be extended to complete $\cB^*$;
\item 
Let $x^*$ be an optimal solution such that $A(:,\supp(x^*))$ is full
row rank.
Then, $\supp(s^*)=\cN^*$ and $\supp(x^*)$ can be extended to complete $\cB^*$.\end{enumerate}
Suppose that we are given a primal-dual optimal solution $(x^*,y^*,s^*)$
of the form 
\begin{equation}
\label{eq:optimalPDpair}
\begin{bmatrix} A_\cB & A_\cJ & A_\cN \end{bmatrix} 
\begin{pmatrix} x_\cB \\ x_\cJ \\ x_\cN \end{pmatrix}
= b, \text{ where } 
\rank(A_\cB) = m, \
\begin{pmatrix} x_\cB \\ x_\cJ \\ x_\cN \end{pmatrix}
\begin{array}{c} > \\ = \\ = \end{array}
\begin{pmatrix} 0 \\ 0 \\0 \end{pmatrix}  \text{ and }
\begin{pmatrix} s_\cB \\ s_\cJ \\ s_\cN \end{pmatrix}
\begin{array}{c} = \\ = \\ > \end{array}
\begin{pmatrix} 0 \\ 0 \\0 \end{pmatrix}  .
\end{equation}
We claim that $\cN^* = \supp(s^*)$.
That is, the support of the current dual optimal solution $s^*$ is maximal and hence we obtain the strict complementary partition for free. 
We rewrite the system $Ax=b$ of \cref{eq:optimalPDpair} as
\[
\begin{bmatrix} A_{\cB_1} & A_{\cB_2} & A_\cJ  \end{bmatrix} 
\begin{pmatrix} x_{\cB_1} \\ x_{\cB_2} \\ x_\cJ \end{pmatrix} = b, \
\text{ where } 
A_{\cB} =  \begin{bmatrix} A_{\cB_1} & A_{\cB_2}  \end{bmatrix} , \ 
x_\cB = \begin{pmatrix} x_{\cB_1} \\ x_{\cB_2} \end{pmatrix}
\text{ and }   \rank(A_{\cB_1}) = m.
\]
Then, by replacing the data in \Cref{algo:Nondegen2Slater} by 
\[
\cN \leftarrow \cJ,  \ A\leftarrow A(:,\cB_1\cup \cB_2 \cup \cN) , \ \tilde{x} \leftarrow x^* ,
\]
we can endow positive weights to $x_\cJ$ while maintaining the primal feasibility.
Since we maintain the feasibility of the primal-dual solution without violating the complementarity, we maintain the optimality.

\subsubsection{Lack of Strict Feasibility and Interior Point Methods}
\label{sec:illCondIPTIPS}

In this section we provide a new perspective on the ill-conditioning
that typically arises in interior point methods. 
Many interior point algorithms are derived from block
Gaussian-elimination of the linearized primal $(\cP$) and dual $(\cD$)
optimality conditions (KKT conditions).
Let $(x_c,y_c, s_c)$ be the current primal-dual pair iterate.
The search direction is computed by solving the Newton equation
\begin{equation}
\label{eq:NewtonSystemIPM}
\begin{bmatrix}
& A^T & I \\
A & & \\
\Diag(s_c) & &\Diag(x_c)
\end{bmatrix} 
\begin{pmatrix}
\Delta x \\ \Delta y \\ \Delta s
\end{pmatrix} 
= 
-\begin{pmatrix} r_d \\ r_p \\ r_c \end{pmatrix} ,
\end{equation}
where $r_d, r_p, r_c$ are the residuals 
of dual feasibility, primal feasibility and complementarity, respectively. 
After the block elimination, we first
find the change $\Delta y$ by solving the so-called 
normal equation, a square system,
\begin{equation}
\label{eq:normalMatIPT}
AD_cA^T \Delta y = \bar{r} , \text{ where } 
D_c = \Diag(x_c) \Diag(s_c)^{-1}, 
\end{equation}
$\bar{r}\in \Rm$ is some residual; see e.g., \cite[Chapter 11]{{SWright:96}}.
It is known that \eqref{eq:normalMatIPT} often becomes ill-conditioned near an optimum.
The ill-conditioning of the matrix $AD_cA^T$ under degeneracy is
discussed in \cite{MR94j:90021}
in terms of the lack of nice positive diagonal elements of $D_c$. 
This relates to our results in the sense that 
all vertices that form the optimal face of $(\cP)$ are also degenerate in the absence of strict feasibility.
Moreover, we show that the ill-conditioning of the matrix $AD_cA^T$ not only originates from the columns of $A$ chosen by $D_c$ but also from the rows of $A$ in the absence of strict feasibility. 
In particular, a large $\ips$ is a good indicator for ill-conditioning.

We partition the matrix $A =  \begin{bmatrix}
P_{\bar{m}}AV & A_{\cI_0} \\ R_{AV} & R_{\cI_0} \end{bmatrix}$, where $[A_{\cI_0};R_{\cI_0}]$ corresponds to the submatrix of $A$ associated with the index set $\cI_0$.
The submatrix $R_{AV}$ refers to the rows of $A$ that are implicitly redundant due the lack of strict feasibility.
Let $(x^*,y^*,s^*)$ an optimal triple and let $D^* = \Diag(x^*) \Diag(s^*)^{-1}$.
As $x_c \to x^*$, i.e., as the iterates get closer to the feasible set $\cF$, we observe the limiting behaviour $A D_c A^T  \to AD^*A^T$ below:
\[
\begin{array}{cccl}
A D_c A^T 
\  \to   \  &
AD^*A^T  &=& 
\begin{bmatrix}
P_{\bar{m}}AV & A_{\cI_0} \\ R_{AV} & R_{\cI_0}
\end{bmatrix}
\begin{bmatrix}
D^*_{AV} & 0 \\ 0& 0
\end{bmatrix}
\begin{bmatrix}
P_{\bar{m}}AV & A_{\cI_0} \\ R_{AV} & R_{\cI_0}
\end{bmatrix}^T \\
&&=&
\left[
\begin{array}{rr}
(P_{\bar{m}}AV) D^*_{AV} (P_{\bar{m}}AV)^T  & (P_{\bar{m}}AV) D^*_{AV} R_{AV}^T \\
R_{AV} D^*_{AV} (P_{\bar{m}}AV)^T & R_{AV}  D^*_{AV} R_{AV}^T
\end{array}
\right]
\end{array} 
\]
where $D^*_{AV}$ is the submatrix of $D^*$  with the diagonal associated with $\cI_+$. 
We recall from \Cref{lemma:ConstrRedundant} that the rows of $R_{AV}$ are linear combinations of the rows of $P_{\bar{m}}AV$. Therefore, the more implicit redundant constraints $\cF$ has, 
the more `$0$' singular values $AD^*A^T$  has, i.e.,~ill-conditioned.

The self-dual embedding \cite{{int:Ye55}} is a popular formulation of the primal-dual \LP pair used for an interior point method. An attractive feature of the self-dual embedding is that a \emph{feasible} initial iterate in the interior is analytically given. 
The success of the self-dual embedding technique is supported by strong performances of some solvers.
However, the absence of strict feasibility results in the same type of ill-conditioning even when this reformulation is used.
For instance, \cite[equation (17)]{int:Ye55} displays the equation as a part of computing the search direction~$(d_x;d_y)$:
\[
\begin{bmatrix}
X^k S^k & - X^k A^T \\  A X^k & 0 
\end{bmatrix} 
\begin{pmatrix} (X^k)^{-1} d_x \\ d_y \end{pmatrix}
= 
\begin{pmatrix} \gamma \mu^k e - X^k s^k \\ 0  \end{pmatrix}
- \begin{bmatrix} X^k c  & - X^k \bar{c} \\ -b & \bar{b} \end{bmatrix}
\begin{pmatrix} d_\tau \\ d_\theta \end{pmatrix} .
\]
Here, $X^k = \Diag(x^k)$ and $S^k = \Diag(s^k)$, where $x^k,s^k$ are the current primal-dual iterate. 
It then uses the back-solve steps to complete the remaining components of the search direction.
For simplicity, we set the right-hand side of the system to be $\begin{pmatrix}
r_1 \\ r_2 \end{pmatrix}$.
By expanding the first block equation, we obtain
\[
(X^k S^k) (X^k)^{-1} d_x   - X^k A^T d_y = r_1 
\iff
(X^k)^{-1} d_x = (X^k S^k)^{-1}r_1 +  (X^k S^k)^{-1} X^k A^T d_y .
\]
We then substitute the equality above into the second block equation, i.e.,
\[
\begin{array}{rcl}
A X^k (X^k)^{-1} d_x  = r_2 
&\iff & A X^k (S^k)^{-1} A^T d_y  = r_2 -A X^k(X^k S^k)^{-1}r_1 . \\
\end{array}
\]
Finally, we obtain the normal matrix $A X^k (S^k)^{-1} A^T$ that appear in \cref{eq:normalMatIPT}.

\subsubsection{Lack of Strict Feasibility in the Dual}
\label{sec:NoSlaterDual}

Recall~\Cref{rem:addredconstr} that redundant constraints can result in
poor behaviour for interior point methods. Moreover, complementary
slackness means we get dual variables fixed at $0$. This is one motivation for
considering \FR on the dual $(\cD)$; see \cref{eq:dualLP}.
We denote the feasible set of the dual $(\cD)$ by
\begin{equation}
\label{eq:dualFeaSet}
\cG := \{ (y,s) \in \Rm \oplus \Rnp \ : \ A^Ty + s = c \}
=
\left\{(y,s) \in \Rm \oplus \Rnp \ 
: \ \begin{bmatrix} A^T & I \end{bmatrix}
\begin{pmatrix} y \\ s \end{pmatrix} =c \right\} .
\end{equation}
The facial reduction arguments applied to the dual are parallel to the
ones given in \Cref{sec:FacialReduction}. 
We provide the theorem of the alternative 
for the dual and a short 
derivation for the facially reduced system for $\cG$ in \Cref{sec:dualdegennoSlater}.
We also conclude
that the absence of strict feasibility for $\cG$ implies dual
degeneracy at all \BFSp s.

A popular method for rewriting an instance with a free variable $x_i$ into the primal standard form
is to write $x_i$ into the difference of two nonnegative variables, i.e.,~$x_i=x_i^+ - x_i^-$ with $x_i^+, x_i^- \ge 0$.
This equivalent transformation does not seem to cause any difficulties at first glance;
at least the primal simplex method 
does not consider both $x_i^+$ and $x_i^-$ as  a basic variables simultaneously in order to form a nonsingular basis matrix. 
However, this equivalent transformation has a significant consequence to the dual program.
For any  $K \ge \max\{ x_i^+,x_i^- \}$, we can maintain the equality
\[
x_i = x_i^+ - x_i^- =(x_i^++K) - (x_i^-+K) . 
\] 
Thus, the primal optimal set is unbounded.
This implies that the dual feasible region of the reformulated primal does not have a strictly feasible point.
Consequently, the results that we established for the primal applies to the dual; (i) 
this implies that all \BFSp s of the dual are degenerate; (ii) the equality system for the dual feasibility contains implicit redundancies and thus the Newton equation that appear in the interior point method  \eqref{eq:NewtonSystemIPM} becomes very ill-conditioned near an optimum. 
More details for loss of strict feasibility in the dual is given 
in~\Cref{sec:dualDegenSection}.

\section{Numerical Investigation}
\label{sec:Numerics}
We now provide empirical evidence that \FR is indeed a useful
preprocessing tool for reducing the size of problems as well as for
improving the \emph{conditioning}. We do this
first for interior point methods and then for simplex methods.
In particular, this provides empirical evidence that lack of strict
feasibility is equivalent to implicit singularity.
All the numerical tests are performed using MATLAB version 2021a on Dell XPS 8940 with 11th Gen Intel(R) Core(TM) i5-11400 @ 2.60GHz 2.60 GHz with 32 Gigabyte memory.     
We use three different solvers in our tests: (i) \textdef{linprog} from
MATLAB\footnote{\url{https://www.mathworks.com/}. Version 9.10.0.1669831
(R2021a) Update 2.};
(ii) \textdef{SDPT3}\footnote{\url{https://www.math.cmu.edu/~reha/sdpt3.html}, version
SDPT3 4.0.}; and (iii) \textdef{MOSEK}\footnote{\url{https://www.mosek.com/}. 
Version 8.0.0.60.}.
MATLAB version 2021a is used to access all the solvers for the tests, and we
use their default settings for stopping criteria. 
Note that MOSEK has a preprocessing option.\footnote{MOSEK has a \href{https://docs.mosek.com/latest/toolbox/presolver.html}{presolve} with five steps that includes eliminating fixed variables. However, itis clear from the empirical evidencethat the variables fixed at $0$ are not found.}

\subsection{Empirics with Interior Point Methods}
\label{sec:NumericIntPtMethod}

In this section we compare the behaviour for finding near-optimal points 
with instances that do and do not satisfy strict feasibility.
More specifically, given a near optimal primal-dual point $(x^*,s^*)\in
\Rnpp \oplus \Rnpp$ obtained from an interior point solver, 
we observe the condition number, i.e.,~the ratio of largest
to smallest eigenvalues of the normal matrix at $(x^*,s^*)$:
\begin{equation}
\label{eq:condNumnearOpt}
\kappa\left( AD^*A^T \right) , \ \text{ where } D^* = \Diag(x^*)\Diag(s^*)^{-1} .
\end{equation}
We show that instances that do not have strictly feasible points tend 
to have significantly larger condition numbers of the normal 
equation near the optimum. We also present a numerical experiment 
on perturbations of the right-hand side vector $b$.

\subsubsection{Generating \LPp s without Strict Feasibility}
\label{sec:GenerationPrimal}

Given $m,n,r\in \N$, we construct the data $A \in
\Rmn$ and $b \in \Rm$ to satisfy \cref{eq:auxsystem} with $r$ as the dimension of the relative interior of $\cF$, $\relint (\cF)$.
\begin{enumerate}
\item
Pick any $0\neq y \in \Rm$.
Let 
\[
\{ y \}^\perp =\spann \{a_i\}_{i=1}^{m-1}\quad ( = \nul( y^T)).
\]
We let $R\in \R^{(m-1)\times r}$ be a random matrix, and get
\[
A_1 := \begin{bmatrix} a_1&\ldots& a_{m-1}\end{bmatrix}\!R
 \in \R^{ m \times r}, \quad 
A_1^Ty = 0 \in \R^r.
\]

\item
Pick any $\hat{v} \in \R^r_{++}$ and set $b = A_1 \hat{v}$.
We note that $y^TA_1 = 0$ and $\<b,y\> = 0$.

\item
Pick any matrix $A_2 \in \R^{m\times (n-r)}$ satisfying $(y^T A_2)_i \ne 0, \ \forall i$.
If there exists $i$ such that $(y^TA_2)_i<0$, then change the sign of
the $i$-th column of $A_2$ so that we conclude
\[
(A_2^Ty)  \in \R^{n-r}_{++}.
\]

\item
We define the matrix  $A = \begin{bmatrix} A_1 & A_2\end{bmatrix}  \in \Rmn$.
Then $\{x \in \Rnp : Ax=b\}$ is a polyhedron with a feasible point
$\hat{x} = [\hat{v}; 0 ]$ having $r$ number of positives. The vector $y$
is a solution for the system \cref{eq:auxsystem}:
\[
0\lneqq z=A^Ty 
= \begin{pmatrix} A_1^Ty = 0 \cr A_2^Ty>0\end{pmatrix}, \, b^Ty=0.
\]
We then randomly permute the columns of $A$ to avoid the zeros always
being at the bottom of the feasible variables $x$.
\end{enumerate}

For the empirics, we construct the objective 
function $c^Tx$ of $(\cP)$ as follows.
We choose any $\bar{s} \in \R^n_{++}, \bar{y} \in \Rm$ and set $c = A^T\bar{y} +\bar{s}$. 
Then we have the data for the primal-dual pair of \LPp s  and the primal \emph{fails} strict feasibility:
\[
(\cP_{(A,b,c)})  \quad \min \{ \ c^Tx : Ax = b, \ x\ge 0 \ \}\quad \text{ and } \quad
(\cD_{(A,b,c)})  \quad \max \{ \ b^Ty : A^Ty + s = c, \ s \ge 0 \ \}.
\]
We note that by choosing $\bar{s} \in \R^n_{++}$, the dual problem $(\cD_{(A,b,c)})$ has a strictly feasible point. 
In order to generate instances with strictly feasible points, we
maintain the same data $A,c$ used for the pair $(\cP_{(A,b,c)})$ and
$(\cD_{(A,b,c)})$. We only redefine the right-hand side vector by $\bar{b} = Ax^\circ$, where $x^\circ \in \R^n_{++}$:
\[
(\bar{\cP}_{(A,\bar{b},c)})  \quad \min \{ \ c^Tx : Ax = \bar{b}, \ x\ge 0 \ \}\quad \text{ and } \quad
(\bar{\cD}_{(A,\bar{b},c)})  \quad \max \{ \ \bar{b}^Ty : A^Ty + s = c, \ s \ge 0 \ \}.
\]
The facially reduced instances of $(\cP_{(A,b,c)})$ are denoted by $(\cP_{(A_{FR},b_{FR},c_{FR})})$. They are obtained by discarding the variables that are identically $0$ in the feasible set $\cF$ and the redundant constraints.
In other words, the affine constraints of $(\cP_{(A_{FR},b_{FR},c_{FR})})$ are of the form \cref{eq:setEqauivalencefullrow}.

\subsubsection{Condition Numbers}

In order to illustrate the differences in condition numbers of the
normal matrices, we solve the three families of instances:
\\(i) $(\cP_{(A,b,c)})$,  strictly feasible fails; (ii)
$(\bar{\cP}_{(A,\bar{b},c)})$, strictly feasible holds;
(iii) $(\cP_{(A_{FR},b_{FR},c_{FR})})$,  facially reduced instances of $(\cP_{(A,b,c)})$.


\begin{figure}[h!]
\centering
\includegraphics[height=6cm]{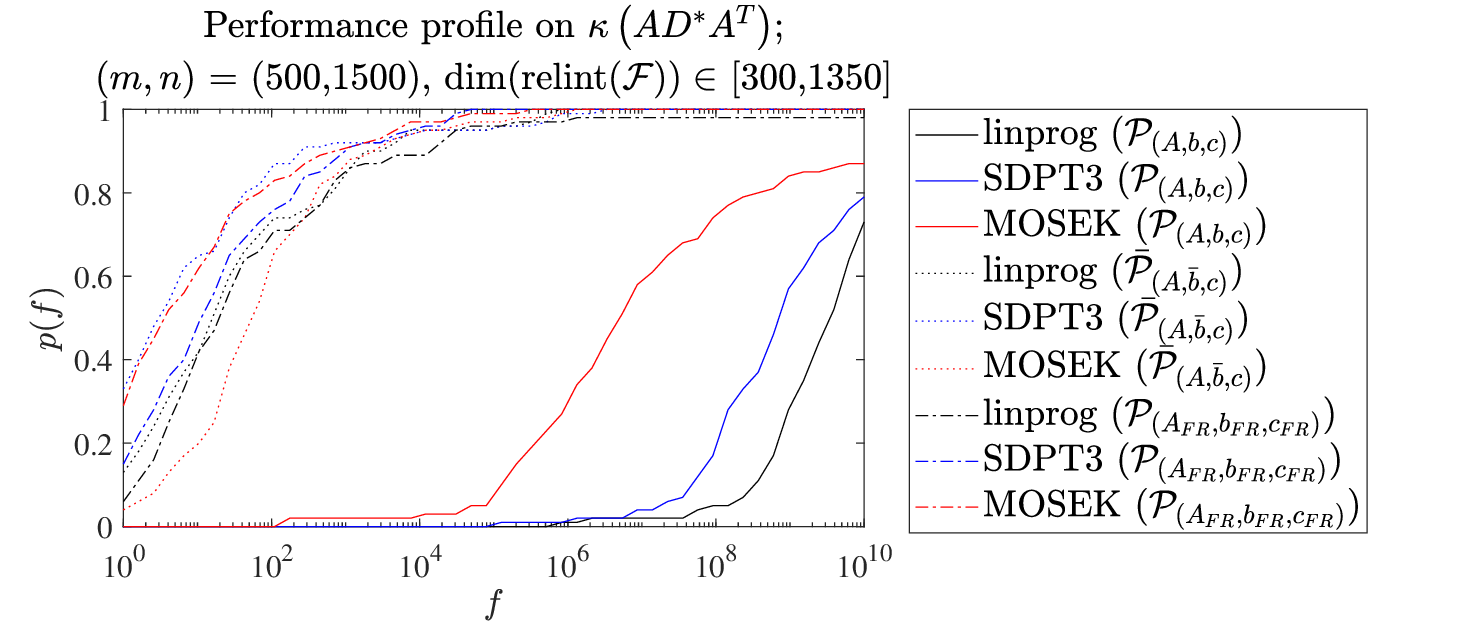}
\caption{Performance profile on $\kappa\left(AD^*A^T\right)$ with(out) strict
feasibility near optimum; various solvers }
\label{fig:condnums}
\end{figure}

In \Cref{fig:condnums} we use a \textdef{performance profile} 
\cite{MR1875515,GouldNicholas2016ANoP} to observe the overall behaviour on different families of instances using the three solvers.
The performance profile provides a useful graphical comparison for solver performances.
\Cref{fig:condnums} displays the performance profile on the condition numbers of the normal matrix $AD^*A^T$ near optimal points from different solvers. We generate $100$ instances for each family that have $\dim(\relint(\cF))\in [300,1350]$. The instance sizes are fixed with $(m,n) = (500,1500)$.
The vertical axis in~\Cref{fig:condnums} represents the statistics  of the performance ratio on $\kappa\left(AD^*A^T\right)$, the condition number of normal matrix near optimum $(x^*,s^*)$; see \cref{eq:condNumnearOpt}. 
Roughly speaking, the vertical axis represents the probability of achieving a performance ratio within a factor of $f$ among all methods used. We used the lower the better statistics.  
The details of the performance ratio are discussed in~\cite{MR1875515,GouldNicholas2016ANoP}.
The solid lines in \Cref{fig:condnums} represent the performance of the instances $(\cP_{(A,b,c)})$ that fail strict feasibility.  
They show that the condition numbers of the normal
matrices near optima are significantly higher when strict feasibility fails.
That is, when strict feasibility fails for $\cF$, the matrix $AD^*A^T$ is more ill-conditioned and it is difficult to obtain search directions of high accuracy.
We also observe that facially reduced instances yield smaller condition numbers near optima. We note that the instances $(\cP_{(A,b,c)})$ and $(\cP_{(A_{FR},b_{FR},c_{FR})})$ are equivalent.


\subsubsection{Stopping Criteria}

We now use the three solvers to observe the accuracy of the
first-order optimality
conditions (KKT conditions) and the running times, for the instances
$(\cP_{(A,b,c)})$ and $(\cP_{(A_{FR},b_{FR},c_{FR})})$, see
\Cref{table:KKTtable}.
We test the average performance of $10$ instances of the size $(n,m,r) = (3000,500,2000)$.
The headers used in \Cref{table:KKTtable} provide the following.
Given solver outputs $(x^*,y^*,s^*)$, the header `KKT' exhibits the
average of the triple consisting of the primal feasibility, dual feasibility 
and complementarity;
\[
\text{KKT} = \left( \frac{\|Ax^*-b\|}{1+ \|b\|}, \ \frac{\|A^T y^*+s^*-c\|}{1+ \|c\|} , \ \frac{\<x^*,s^*\>}{n} \right) .
\]
The headers `iter' and `time'  in \Cref{table:KKTtable} refer to
the average of the number of iterations and the running time in seconds, respectively.
 
\begin{table}[h!]
\centering
\begin{tabular}{|c|c|c|c|}\hline 
\multicolumn{2}{|c|}{\multirow{1}{*}{  }} & \multicolumn{1}{c|}{Non-Facially Reduced System} & \multicolumn{1}{c|}{Facially Reduced System}  \\ \cline{1-4}
\multirow{3}{*}{ linprog }
&\multirow{1}{*}{KKT} & (2.44e-15, 2.05e-12, 3.18e-09) & (5.85e-16, 4.74e-16, 9.22e-09) \\ \cline{3-4} 
&\multirow{1}{*}{iter} & 22.30 & 17.90 \\  \cline{3-4} 
&\multirow{1}{*}{time} &     2.34 &     0.81 \\ \cline{1-4} 
\multirow{3}{*}{ SDPT3 }
&\multirow{1}{*}{KKT} & (8.11e-10, 7.55e-12, 5.65e-02) & (1.43e-11, 3.67e-16, 4.38e-06) \\ \cline{3-4} 
&\multirow{1}{*}{iter} & 25.50 & 19.30 \\  \cline{3-4} 
&\multirow{1}{*}{time} &     1.73 &     0.70 \\ \cline{1-4} \multirow{3}{*}{ mosek }
&\multirow{1}{*}{KKT} & (7.52e-09, 1.80e-15, 3.27e-06) & (3.85e-09, 3.69e-16, 1.19e-06) \\ \cline{3-4} 
&\multirow{1}{*}{iter} & 40.30 & 10.20 \\  \cline{3-4} 
&\multirow{1}{*}{time} &     1.40 &     0.35 \\ \cline{1-4} 
\end{tabular}
\caption{Average of KKT conditions, iterations and time of (non)-facially reduced problems}
\label{table:KKTtable}
\end{table}

From \Cref{table:KKTtable} we observe that facially reduced instances
provide significant improvement in first-order optimality conditions, 
the number of iterations and the running times for all solvers in general.
We note that the instances $(\cP_{(A,b,c)})$ and
$(\cP_{(A_{FR},b_{FR},c_{FR})})$ are equivalent. Hence, our empirics
show that  performing facial reduction as a preprocessing step not only improves the solver running time but also the \emph{quality} of solutions.

\subsubsection{Distance to Infeasibility}
\label{sec:numericsDistInfes}

In this section we present empirics that illustrate the effect of 
perturbations of the right-hand side $b$ when strict feasibility fails.
We recall, from \Cref{prop:distInfrhsPert}, that there exists an arbitrarily small perturbation of the right-hand side vector $b$ of $\cF$ that renders the set $\cF$ infeasible, i.e., $\dist(b,\cF=\emptyset)=0$.  
Moreover, the vector $\Delta b = y$ that satisfies the auxiliary system \cref{eq:auxsystem} is a perturbation that makes the set $\cF$ empty; see \cref{eq:FarkasInfea}.

We follow the steps in \Cref{sec:GenerationPrimal} to generate instances of the order $(n,m)=(1000,200)$ and $r = \relint(\cF) = 900$.
The objective function $c^Tx$ is chosen as presented in \Cref{sec:GenerationPrimal}.
For the fixed $(n,m,r)$, we generate $10$ instances and observe the 
average performance of these instances as we gradually increase the magnitude of the perturbation. 
We recall the matrix $AV$ from \cref{eq:setEqauivalence}.
We use two types of perturbations for $b$;
\[
\Delta b, \text{ where } \Delta b \in \range(AV)^\perp, \quad
\Delta \bar{b}, \text{ where } \Delta \bar{b}\in \range (AV).
\]
We choose $\Delta b$ to be the vector $y$ that satisfies $\cref{eq:auxsystem}$.
For $\Delta \bar{b}$, we choose $AV d$, where $d\in \R^r$ is a randomly chosen vector.
As we increase $\epsilon>0$, we observe the performance of the two families of the systems
\[
\begin{array}{lll}
Ax = b_\epsilon := b - \epsilon \Delta b \ \text{ and } \ Ax = \bar{b}_\epsilon := b - \epsilon \Delta \bar{b} .
\end{array}
\]
We use the interior point method from MATLAB's linprog for the test. 
\Cref{fig:firstoptcond} contains the average of the first-order optimality conditions evaluated at the solver outputs $(x^*,y^*,s^*)$ of these instances; primal feasibility, dual feasibility and the complementarity.

\begin{figure}[ht!]
\centering
\includegraphics[height=6cm]{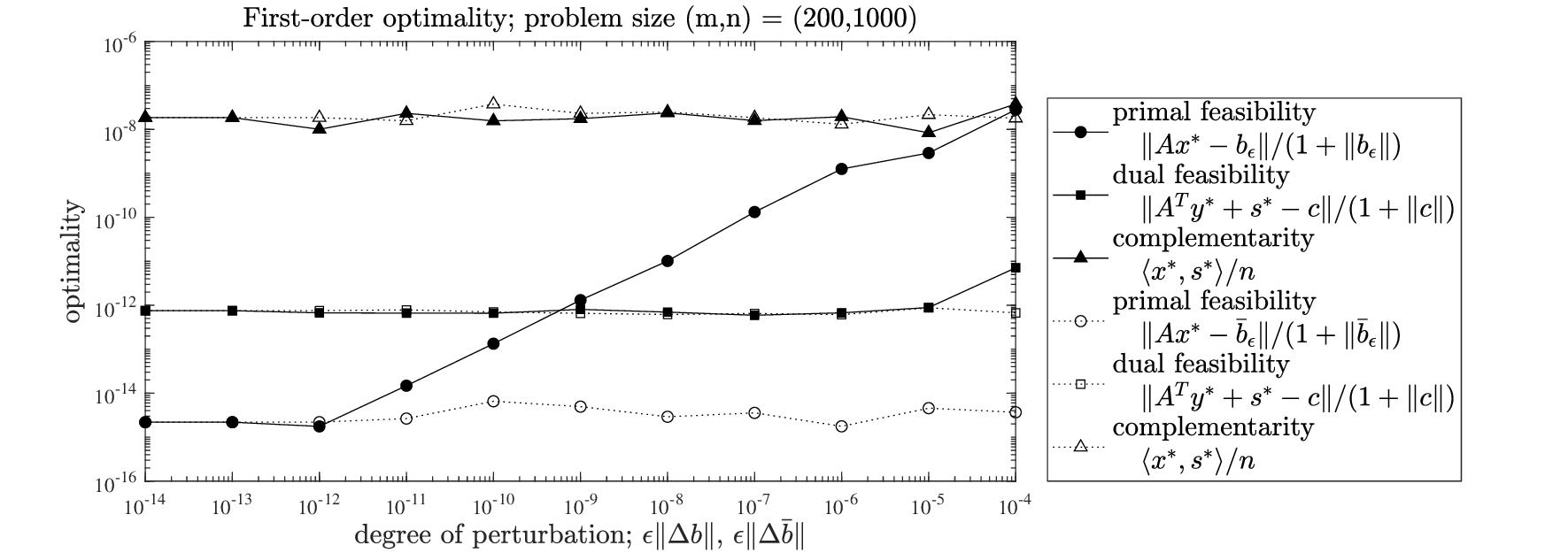}
\caption{Changes in the first-order optimality condition as the perturbation of $b$ increases}
\label{fig:firstoptcond}
\end{figure}
The horizontal axis of \Cref{fig:firstoptcond} indicates the degree of the perturbation imposed on the right-hand side vector $b$, $\epsilon \| \Delta b\|$ and $\epsilon \|\Delta \bar{b}\|$.
The vertical axis indicates the individual component of the first-order optimality. 
From \Cref{fig:firstoptcond}, we observe that the KKT conditions with the perturbation $\Delta \bar{b}$ display a steady performance regardless of the perturbation degree;
see the markers  $\circ$,\text{\scriptsize$ \square, \triangle $} with the dotted lines.
In contrast, the markers $\bullet$,\text{\scriptsize$ \blacksquare, \blacktriangle $} in \Cref{fig:firstoptcond} exhibit the performance of the instances that are perturbed with $\Delta b$ and they display a different performance.
In particular, we see that the relative primal feasibility $\|Ax^*-b_\epsilon \|/(1+\|b_\epsilon \|)$, marked with $\bullet$, consistently increases as the perturbation magnitude $\epsilon \|\Delta b\|$ increases when strict feasibility fails for $\cF$.

\subsubsection{Empirics on Singular Values and $\ips$}
\label{sect:SVDips}

In this section we present our numerical experiment on the ill-conditioning discussed in \Cref{sec:illCondIPTIPS} in terms of $\maxsd$ (see \Cref{def:ips}). 
We generated instances with different settings for $\maxsd = 1,5$ and $10$.
We recall the generation for the vector $y$ and $A_2$ in \Cref{sec:GenerationPrimal}.
For generating and instance with $\maxsd >1$, 
we generated $Y_c = \blkdiag(y^1,\ldots, y^{\ips}) \in \R^{m\times \maxsd}$ and $A_2 = \blkdiag(A_2^1, \ldots, A_2^{\maxsd}) $
of appropriate dimension 
in order to produce the exposing vector $A_2^T \sum_{j=1}^{\maxsd} Y_c(:,j) \ge 0$.
Each column of $Y_c$ serves as a vector satisfying \cref{eq:auxsystem}.

Let $\sigma_{\max}(AD^*A^T)$ be the maximum singular value of $AD^*A^T$.
We count the number of singular values of $AD^*A^T$ that are smaller than $10^{-8}\cdot \sigma_{\max}(AD^*A^T)$.
In \Cref{table:ipsEmpirics} below, we report the cardinality of
\[
\textdef{$\Sigma_0 :=  \{ i : \sigma_i (AD^*A^T)
<\sigma_{\max}(AD^*A^T)\}$}.
\]
We test the average performance on the $20$ instances of the fixed size $(n,m,r)=(3000,500,2000)$.
We display the average number of $| \Sigma_0|$.
We see from \Cref{table:ipsEmpirics} 
\begin{table}[h!]
\centering
\begin{tabular}{|c|c|c|c|c|}\hline 
\multicolumn{2}{|c|}{\multirow{1}{*}{  }} 
& \multicolumn{1}{c|}{ $\maxsd$ = 1} & \multicolumn{1}{c|}{ $\maxsd$ = 5} & \multicolumn{1}{c|}{ $\maxsd$ = 10} \cr \cline{1-5}
\multirow{1}{*}{ linprog }
&\multirow{1}{*}{$| \Sigma_0 |$} & 4.10 & 8.65  & 13.10  \cr \cline{1-5}
\multirow{1}{*}{ SDPT3 }
&\multirow{1}{*}{$| \Sigma_0 |$} & 4.75 & 8.00  & 34.65  \cr \cline{1-5}
\multirow{1}{*}{ MOSEK }
&\multirow{1}{*}{$| \Sigma_0 |$} & 6.45 & 12.35  & 14.50  \cr \cline{1-5}
\end{tabular}
\caption{$\#$ (rel.) small singular values of $AD^*A^T$
near optimum; average over $20$ instances }
\label{table:ipsEmpirics}
\end{table}
a larger $\maxsd$ and $\ips$ values produce a greater number of small singular values.
When there is a significant number of redundant constraints, it is more difficult to obtain a good search direction due to a large number of relatively small singular values.



\subsection{Empirics with Simplex Method}

In this section we compare the behaviour of the dual simplex method with 
instances that have strictly feasible points and instances that do not.
We also observe the degeneracy issues that arise in the instances from \href{https://www.netlib.org/lp/}{NETLIB}.

\subsubsection{Empirics on the Number of Degenerate Iterations}
\label{sec:empiricsDegenIter}

In this section we test how the lack of strict feasibility affects the
performance of the dual simplex method. 
We provide the construction of instances that fail strict feasibility in \Cref{sec:GenerationDual}.
We choose MOSEK for our tests since MOSEK reports the percentage of degenerate iterations as a part of the solver report. 
MOSEK reports the quantity `DEGITER($\%$)', the ratio of degenerate iterations.

Given a set $\cG$ and a point $(y,s)\in \relint(\cG)\subseteq \Rm \oplus \Rn_+$, let $r$ be the number of positive entries of $s$, i.e., $r = |\supp(s)|$. 
In our tests, we gradually increase $r$ for fixed $n,m$ and generate instances for $\cG$ as described in \Cref{sec:GenerationDual}. We then observe the behaviour of the dual simplex method. 
\Cref{table:dualsimplex} contains the results.
In \Cref{table:dualsimplex}, a smaller value for the header $(r/n)\%$
means that there are more entries of $s$ that are identically $0$ in the set $\cG$; and the value $0\%$ means that strict feasibility holds. 
For each triple $(n,m,r)$, we generated $10$ instances and we report the average 
of `DEGITER($\%$)' of these instances.

\begin{table}[h!]
\centering
\begin{tabular}{|c|c|ccccc|}\hline 
\multicolumn{2}{|c|}{  } & 
\multicolumn{5}{|c|}{ $100\%-(r / n) \% $ } \\
\cline{3-7}
\multicolumn{2}{|c|}{  } &
40 & 30 & 20 & 10 & 0 
\\ \hline
\multirow{ 4 }{*}{ $(n,m)$ }
& \multirow{1}{*} (1000, 250) &36.62& 10.18& 0.01& 0.02& 0.00\\ 
& \multirow{1}{*} (2000, 500) &39.72& 18.28& 0.07& 0.15& 0.01\\ 
& \multirow{1}{*} (3000, 750) &25.99& 10.66& 0.32& 0.75& 0.02\\ 
& \multirow{1}{*} (4000, 1000) &29.78& 18.25& 0.25& 0.53& 0.02\\ 
\hline\end{tabular}
\caption{Average of the ratio of degenerate iterations}
\label{table:dualsimplex}
\end{table}

We recall \Cref{thm:LPdegen}: lack of strict feasibility implies that
all basic feasible solutions are degenerate. 
However, we observe more, i.e.,~from \Cref{table:dualsimplex}, 
the frequency of
degenerate iterations increases as $r$ decreases. In other words, higher
degeneracy of the set $\cG$ yields more degenerate iterations when the
dual simplex method is used.


\subsubsection{NETLIB Problems; Perturbations; Stability} 
\label{sec:NetlibExperiment}

We now illustrate the lack of strict feasibility on instances from the
\href{https://www.netlib.org/lp/}{NETLIB} data set. We used the following first $67$ instances that are in standard form at \href{http://users.clas.ufl.edu/hager/coap/format.html}{this link}:
\[
\begin{scriptsize}
\begin{tabular}{llllllllllllllll}
25fv47      &
adlittle$^*$ &
afiro       &
agg$^*$     &
agg2$^*$    &
agg3$^*$    &
bandm$^*$   &
beaconfd$^*$&
blend       &
bnl1$^*$    \\
bnl2$^*$    &
brandy$^*$  &
cre$\_$a$^*$ &  
cre$\_$b$^*$  &
cre$\_$c$^*$    & 
cre$\_$d$^*$     &
d2q06c$^*$    &
degen2$^*$    &
degen3$^*$    &
e226$^*$      \\
fffff800$^*$  &
israel        &
lotfi         &
maros$\_$r7   &   
nug05         &
nug06         &
nug07         &
nug08       &
nug12         &
nug15         \\
nug20         &
osa$\_$07$^*$    &
osa$\_$14$^*$    &
qap12         &
qap15         &
qap8          &
sc105$^*$     &
sc205$^*$     &
sc50a$^*$     &
sc50b$^*$     \\
scagr25       &
scagr7      &
scfxm1$^*$    &
scfxm2$^*$    &
scfxm3$^*$    &
scorpion$^*$  &
scrs8$^*$     &
scsd1         &
scsd6         &
scsd8         \\
sctap1        &  
sctap2        &
sctap3        &
share1b       &
share2b       &
ship04l$^*$   &
ship04s$^*$   &
ship08l$^*$   &
ship08s$^*$   &
ship12l$^*$  \\
ship12s$^*$   &
stocfor1      &
stocfor2      &
stocfor3      &
truss         &
wood1p$^*$    &
woodw$^*$     &
\end{tabular}
\end{scriptsize}
\]
We removed redundant rows to guarantee full row rank of $A$.

Surprisingly, the Slater condition fails for $37$ out of these $67$
instances.\footnote{The instances that fail strict
feasibility are marked with an asterisk $*$ in the list above.}
This has interesting implications for both interior point and
simplex methods. The standard interior point method stopping criteria
is complicated by the unbounded dual optimal set.
For the primal simplex method, every iteration is at a
degenerate \BFS and \textdef{stalling} generally occurs. Therefore preprocessing to eliminate the variables
fixed at $0$ is important. In addition, in order to motivate
robust optimization, it is shown in
e.g.,~\cite{MR2546839,MR1702364} that optimal solutions of many of the
NETLIB instances are extremely sensitive to perturbations in the data.
We now see this to be the case, and we show that \FR regularizes
the problem and avoids this instability.

\index{$P_{\bar{m}}AV$}
\index{degree of degeneracy}

We first use the instance \underline{degen3} in order to illustrate the consequence of lack of strict feasibility. 
The data matrix $A$ after removing two redundant rows is $1501$-by-$2604$.
After \FRp, we obtain the constraint matrix $P_{\bar{m}}AV$ of size $1226$-by-$1648$.
This implies that $2604-1648=956$ number of variables are identically
$0$ on the feasible set. Furthermore, $\ips(\cF)=275$ equality constraints 
are implicitly redundant.
By \Cref{item:degDegen1} of \Cref{cor:degrdegAidentz},
without \FRp , the degree of degeneracy of every \BFS is at least $275$.
Namely, the length of the basis is $1501$ and every basis contains at least $275$ degenerate indices.

We now illustrate that \FR gives a more robust model with respect to
data perturbations using the instance \underline{brandy}. 
Let $(A,b)$ be the data after removing the redundant equality constraints.
Let $(P_{\bar{m}}AV,P_{\bar{m}}b)$ be the data for the facially reduced system.
The data matrices $A$ and $P_{\bar{m}}AV$ have sizes $193$-by-$303$ and $155$-by-$260$, respectively\footnote{This also means that, without \FRp, every \BFS has at least $38$ degenerate basic variables. At least $19.69$ percent of basic variables are always degenerate.}.
Set the perturbation scalars $\epsilon_A = \epsilon_b = 10^{-9}$.
We construct a random  perturbation matrix~$\Phi, \| \Phi \|_F =
\|A\|_F+1$, and
random perturbation vector~$\phi, \|\phi\|_2 = \|b\|_2+1$.
We then solve the problem 
\[
\tilde{p}^*= \max \{ \<c,x\> : (A + \epsilon_A \Phi ) x = b+ \epsilon_b \phi, \  x\ge 0 \}.
\]

\index{$P_{\bar{m}}AV$}
\index{$P_{\bar{m}}b$}

For the facially reduced system, we used the identical perturbation data
$\Phi,\phi$ and discard the rows and columns of $(A,b)$ found from \FRp. 
That is, we use the perturbations
$P_{\bar{m}} \Phi  V$ and $P_{\bar{m}} \phi$ to the facially reduced system after the scaling  $\|P_{\bar{m}} \Phi  V\|_F = \|P_{\bar{m}}AV\|_F+1$ and $\|P_{\bar{m}} \phi\|_2 =  \|P_{\bar{m}}b\|_2 +1$. We then solve 
\[
\max \{ \<V^Tc,v\> : (P_{\bar{m}}AV + \epsilon_A P_{\bar{m}} \Phi  V ) v = P_{\bar{m}} b+ \epsilon_b P_{\bar{m}} \phi, \  v\ge 0 \}. 
\]
In this way, we maintain the identical perturbation structure for the original system and the facially reduced system. 
We also generate a transportation problem and use the aforementioned
perturbations. We note that the transportation problems have Slater
points but are known to be highly degenerate. The size of the data generated is $49$-by-$600$.

In the experiment, we tested the instances using $100$ different perturbation settings. 
We randomly generated perturbations $\Phi, \phi$ with density set at
$0.1$. We used MOSEK simplex with the setting `MSK$\_$OPTIMIZER$\_$FREE$\_$SIMPLEX'.
In \Cref{table:netlibBrandy}, the headers $\epsilon_A$ and $\epsilon_b$ refer to the scalars used for perturbations as described above.
The headers $(A,b)$, $(P_{\bar{m}}AV,P_{\bar{m}}b)$ and $(A_{\text{trans}},b_{\text{trans}})$ refer to the non-facially reduced system, the facially reduced system and the transportation problems, with the perturbations.
The integral values in the table indicate the number of times that the
solver  outputs PRIMAL$\_$AND$\_$DUAL$\_$FEASIBLE. 
Let $p^*$ be the optimal value for the unperturbed instance \underline{brandy}, and let $\tilde{p}^*$ be the optimal value of a perturbed instance of \underline{brandy}.
The non-integral values in the table indicate the average relative difference in the optimal values between $p^*$ and $\tilde{p}^*$. 
The relative difference is computed using the formula $\frac{|p^*-\tilde{p}^*|}{2|p^*+\tilde{p}^*|}$.
For example, the first entry $11$ in \Cref{table:netlibBrandy} means that
$100\!-\!11$ out of $100$ perturbed instances yield infeasibility or unknown status, i.e.,~only $11$ solved successfully. The entry 4.938e-02 next to $11$ indicates the average of $\frac{|p^*-\tilde{p}^*|}{2|p^*+\tilde{p}^*|}$ on those $11$ instances.
\begin{table}[h]
\centering
\begin{tabular}{|cccc|c|c|c|}\hline 
\multicolumn{4}{|c|}{ $\epsilon_A$ \qquad $\epsilon_b$  }&  \multicolumn{1}{|c|}{  $(A,b)$ }  &\multicolumn{1}{|c|}{  $( P_{\bar{m}} AV,P_{\bar{m}} b)$  }&$(A_\text{trans},b_\text{trans})$ \\
\hline
 &1.0e-09 & 0   &  &( 11 ,  4.938e-02 ) & ( 97 ,  6.705e-03 ) & 100 \\  
 &0   & 1.0e-09 &  &( 27 ,  2.470e-10 ) & ( 100 ,  2.208e-10 ) & 100 \\  
 &1.0e-09 & 1.0e-09 &  &( 11 ,  1.339e-01 ) & ( 96 ,  8.719e-03 ) & 100 \\  
\hline
\end{tabular}
\caption{Number of successful results out of $100$ perturbed instances using simplex method on the instance \underline{brandy} and transportation problem}
\label{table:netlibBrandy}
\end{table}
The columns $(A,b)$ and $( P_{\bar{m}}
AV,P_{\bar{m}} b)$ in \Cref{table:netlibBrandy} demonstrate that the facially reduced problems are more immune to data perturbations; the number of successfully solved perturbed instances are significantly larger and the optimal values under the perturbations are less influenced.  
The last column indicates that although the instance may have many
degenerate \BFSp s, having a strictly feasible point is important in
terms of perturbations in data, i.e.,~this emphasizes the difference
between the two types of degeneracy.

\section{Conclusion}
\label{sec:Conclusion}

We have addressed the impact, 
for both theoretical and computational reasons, of loss of strict 
feasibility in \LPp, distinguishing one type of degeneracy at a \BFSp. 
For our numerics we illustrated this using the accuracy of optimality
conditions as well as the effect of perturbations, for
the two most popular classes of algorithms, i.e.,~simplex and interior 
point methods.
For the theory, we proved, using the two-step facial reduction, that if
strict feasibility fails for a linear program, then every \BFS is degenerate. In addition, we showed that facial reduction can
be implemented efficiently to obtain a smaller simpler problem with
strict feasibility, and that this improves stability. This was
illustrated on random problems, as well as instances from the NETLIB data
set.

An essential step for almost all algorithms for linear programming
is preprocessing. One part of preprocessing is 
identifying \emph{fixed variables}.
However, identifying variables fixed at $0$, facial reduction, has not
been done due to expense and accuracy problems.
In this paper we have shown that not eliminating these variables, i.e.,~lack
of strict feasibility, is equivalent to implicit singularity and this helps
explain the numerical difficulties that arise.
We have further provided an efficient preprocessing step for
facial reduction, i.e.,~we continue on phase I of the simplex method
that eliminates all the artificial variables, and eliminate the
variables fixed at $0$.  We observed that a variable
that is basic (positive) in every \BFS corresponds to a redundant
constraint and, by complementary
slackness, corresponds to a variable fixed at $0$ in the dual. 
And redundant constraints have been shown in the literature to 
poorly affect algorithms \cite{MR2238662}. Moreover, identifying
\href{https://link.springer.com/chapter/10.1007/978-1-4615-6103-3_13}{redundant
constraints} is a nontrivial operation e.g.,~\cite{MR1482247}.
This motivates
doing \FR on both the primal and the dual problems. (It is still unclear
whether or not we have to repeat \FR on the primal again.)

In the literature, in particular in textbooks on \LPp,
the method most often used to handle a free variable
$x_i$ is to replace it by two nonnegative variables $x_i \leftarrow
x_i^+ - x_i^{-}$. The means that the optimal solution is
unbounded as one can add an arbitrary positive constant to both new
variables. But then strict feasibility fails for the dual, i.e.,~stable
problems are transformed into ill-conditioned problems. One can
speculate that this may account for the large number of instances
in the NETLIB set where strict feasibility fails and numerical accuracy
is difficult to maintain.

We have presented various numerical experiments that convey the
importance of preprocessing for strict feasibility 
for linear programs,~\Cref{sec:Numerics}.
For interior point methods, we illustrated the importance of strict
feasibility using condition numbers and relationships with \emph{nearness
to infeasibility}. 
We also shed light on the main difficulties that 
arose with the implicit redundant constraints and used the QR
decomposition to show how these difficulties come into play.
This also relates to the implicit problem singularity, $\ips$.
A larger $\ips$ means that there is a higher chance of inducing an infeasible problem under perturbations. 
A large number of degenerate \BFSp s  is believed to cause difficulties
for the simplex method. We have shown that the settings for having many
identically $0$ variables in the dual program yield many degenerate
iterations in the simplex method. 
We also have shown that many NETLIB instances fail strict feasibility
and used selected instances to show the effect of this degeneracy.
Moreover, the facially reduced problems are seen to be more robust with respect to data perturbations. 
In addition, an essential element of solving an \LP is
\textdef{postoptimal analysis}, this becomes difficult when strict
feasibility fails and perturbations of $b$ can lead to infeasibility.
These facts further emphasize that ensuring strict feasibility should be part
of preprocessing for linear programming.


Our results can easily extend to other forms of \LPp s and to more
general problems where degeneracies arise, 
such as the active set method for quadratic programs \cite{WolfePhilip1959TSMf,ForsgrenAnders2015Pada}.
We are currently extending the efficient \FR technique to semidefinite
programs.

\section*{Acknowledgements}

This research is supported by the National Sciences and Engineering Research Council (NSERC) of Canada, Grant $\#$ No. 50503-10827.






\newpage




\phantomsection

\cleardoublepage

\appendix
\label{app:all}

\section{Technical Proofs, Supplementary Materials}
\label{sec:appendixProof}

\subsection{proof of \Cref{coro:PatakiLPversion} }
\label{sec:appendixLPataki}

\begin{proof}
Let $x \in F$ and let $r$ be the number of positive entries in $x$. 
Let $\bar{x} \in \R^r$ be the vector obtained by discarding the $0$ entries in $x$. This is readily given by the following matrix-vector multiplication
$\bar{x} = I(\supp(x),:)x$, where $\supp(x)$ is the support of $x$, the set of indices $\{i: x_i>0\}$.
Let $\bar{A} \in \R^{m\times r}$ be the matrix 
after removing the columns of $A$ that are not in the support of $x$,
i.e., $\bar{A} = A_{\supp(x)}$.
We note that $\bar{x} $ is a particular solution to the system $\bar{A} z = b$ and $\bar{x} >0$.

\index{$\supp$, support}
\index{support, $\supp$}

Suppose to the contrary that $r  > m +d$.
Since $r-m>d$, there exists at least $d+1$ linearly independent vectors, say $v_1,\ldots, v_{d+1} \in \R^r$, satisfying 
$\bar{A} v_i = 0,  \ \forall i = 1,\ldots, d+1$.
For each $i \in \{ 1,\ldots, d+1\}$ and for $\epsilon \in \R$, we define 
\[
\begin{array}{ll}
v_{i,+} := \bar{x} + \epsilon v_i , & v_{i,-} := \bar{x} - \epsilon v_i , \\
x_{i,+} := I(:,\supp(x)) \left( \bar{x} + \epsilon v_i \right) ,
& x_{i,-} := I(:,\supp(x)) \left( \bar{x} - \epsilon v_i \right) . \\
\end{array}
\]
For a sufficiently small $\epsilon$, we have $x_{i,+}, x_{i,-} \in \cF$.
We note that $x = \frac{1}{2} (x_{i,+}+ x_{i,-}), \ \forall i$. Hence, by the definition of face, $x_{i,+} \in F, \ \forall i$.
Therefore, $F$ contains vectors $\{x_{i,+}\}_{i=1,\ldots,d+1} \cup \{x\}$ that are affinely independent and hence $\dim (F) \ge d+1$.
\end{proof}

\subsection{A Condition Measure using Degeneracy}
\label{appendix:leastBFS}

Although degeneracy is a well-known subject, to the best of our knowledge, 
the relationships between degeneracy and stability are rarely discussed. 
We now show that the degree of degeneracy at a \BFS provides useful information on the
robustness of the \LPp; the least degenerate \BFS provides an upper bound on the number of implicitly redundant equalities of the set $\cF$. 
We note that an~$\cF$ that contains a large number of implicit redundancies is a more
\emph{ill-conditioned} set.
(This is comparable to a linear system $Ax=b$ with more redundant rows having the error
in the solution being more susceptible to perturbations of $b$.)

The arguments used in the proof of \Cref{cor:degrdegAidentz} are rather algebraic.
The geometric argument used in the proof of \Cref{thm:degeneracyPataki} provides two useful estimates. 
For any extreme point~$x \in \cF$, the number of nonzero elements of $x$, $|\supp(x)|$, satisfies
\[
|\supp(x)| \le m -\ips (\cF) \implies \ips(\cF) \le m - |\supp(x)| .
\]
Since this holds for all extreme points of $\cF$, we get the following:
\begin{equation}
\label{eq:maxsdBoundforLP}
\sd(\cF) \le \maxsd(\cF)\le \ips(\cF) \le \  \hat{d}:= \min_{\text{\BFS $x\in \cF$}} \ \{\text{degree of degeneracy of $x$} \} .
\end{equation}
The shortest \FR steps for $\cF$, $\sd(\cF)$, is at most $1$,  
thus the inequality $\sd(\cF)\le \hat{d}$ does not provide useful information.
However, the relation~\eqref{eq:maxsdBoundforLP} provides two meaningful corollaries related to $\maxsd(\cF)$ and $\ips(\cF)$:
\begin{enumerate}
\item The inequality $\maxsd(\cF) \le \hat{d}$ implies that the number of nontrivial \FR steps cannot exceed the degree of degeneracy of a least degenerate \BFS of $\cF$;
\item The inequality $\ips(\cF)\le \hat{d}$ shows that it is useful to record the minimum degree of degeneracy observed throughout the simplex iterations. 
This gives an estimate for the number of implicitly redundant equalities of $\cF$.
\end{enumerate}
If $\cF$ contains a nondegenerate \BFSp , we get $\hat{d}=0$. 
Hence $\sd(\cF)=\maxsd(\cF)=\ips(\cF)=0$ and it provides an alternative way to view \Cref{coro:contraPosLPdegen}.
We comment that evaluating and recording the degree of degeneracy of a  \BFS are not expensive operations.

\subsection{Dual Degeneracy in the Absence of Strict Feasibility}
\label{sec:dualDegenSection}

\subsubsection{Implicit Redundancies in the Dual}
\label{sec:dualdegennoSlater}

\index{($\cD$), dual of ($\cP$)}
\index{dual of ($\cP$), ($\cD$)}
\index{$\cG$, dual feasible set}
\index{dual feasible set, $\cG$}

The following \Cref{prop:FarkasDual} provides the corresponding dual
form of the theorem of the alternative for
set $\cG$ in \cref{eq:dualFeaSet}.
\begin{lemma}
[theorem of the alternative in dual form,
{\cite[Theorem 3.3.10]{Cheung:2013}}]
\label{prop:FarkasDual}
Let $\cG\neq \emptyset$ in \cref{eq:dualFeaSet}.
Then, exactly one of the following statements holds:
\begin{enumerate}
\item There exists $(y,s) \in \Rm \oplus \Rnpp $ with $A^T y  + s = c $, i.e., strict feasibility holds for $\cG$;
\item There exists  $w \in \Rn$ such that

\begin{equation}
\label{eq:auxsystem:Dual}
0 \ne w \in \Rnp, \ Aw = 0 \ \text{ and } \ \<c,w\> =0 .
\end{equation}
\end{enumerate}
\end{lemma}

\index{$s_w$, support of exposing vector for $\cG$}
\index{support of exposing vector for $\cG$, $s_w$}

We recall that the vector $A^Ty$ in \cref{eq:xexposed} provides an 
exposing vector to the set $\cF$.
Similarly, a solution $w$ to the auxiliary system \cref{eq:auxsystem:Dual}
provides an exposing vector for $\cG$:

\begin{equation*}
\label{eq:sexposed}
(y,s) \in \cG \,\implies \,
\left\{ \<w,s\> = \< w, c - A^T y \> = \<c,w\> - \< Aw,y \> =  
0 - \<0,y\> = 0\right\}. 
\end{equation*}
We let
\[
\cI_w = \{1,\ldots,n\} \setminus \supp(w), \ 
U = I(:,\cI_w)  \  \text{ and } \ s_w = |\supp(w)|.
\]
Then, the facially reduced system of $\cG$ is given by

\begin{equation}
\label{eq:setEqauivalenceDUAL}
\left\{(y,u) \in \Rm \oplus \R_+^{n-s_w} \ : \
\begin{bmatrix} A^T & U \end{bmatrix} 
\begin{pmatrix} y \\ u \end{pmatrix} = c 
\right\} .
\end{equation}

The notion of degeneracy in \Cref{sec:Background} naturally extends to an arbitrary polyhedron, e.g., see~\cite[Section 2]{BT97}.
For a general polyhedron $P \subseteq \Rn$, a point $p$ in $P$ is called a \textdef{basic solution} if there are $n$ linearly independent active constraints at $p$.
In addition, if there are more than $n$ active constraints at the point $p\in P$, then the point $p$ is called \textdef{degenerate}.
Using this definition of degeneracy, we now show that the absence of
strict feasibility for $\cG$ implies that every basic feasible solution of $\cG$ is degenerate.

First, note that the facially reduced system in
\cref{eq:setEqauivalenceDUAL} contains a redundant constraint, i.e.,~let $w$ 
be an exposing vector for $\cG$ from the system
\cref{eq:auxsystem:Dual}. Then we have
\[
\begin{bmatrix} A\\  U^T \end{bmatrix} w
= \begin{bmatrix} Aw \\  U^Tw \end{bmatrix} =
\begin{bmatrix} 0_{m} \\ 0_{n-s_w} \end{bmatrix} .
\]
In other words, there is a nontrivial row combination of
$\begin{bmatrix} A^T & U \end{bmatrix}$ that yields the $0$ vector
implying the existence of a redundant row and
a redundant constraint in the facially reduced system.
The redundancy immediately implies the dual degeneracy; for any basic solution of $\cG$, there always exists an redundant equality in $\begin{bmatrix} A^T & I \end{bmatrix} \begin{pmatrix} y \\ s\end{pmatrix} =c$.

\subsubsection{Construction of Dual \LPp s without Strict Feasibility}
\label{sec:GenerationDual}

We first show how to generate an instance for the dual feasible set
$\cG$ that fails strict feasibility.
The construction is similar to the one in \Cref{sec:GenerationPrimal}.
We generate a degenerate problem by finding a feasible auxiliary system~\cref{eq:auxsystem:Dual}.
Given $m,n,r\in \N$, we construct $A\in \Rmn$ and $c\in \Rn$ that satisfy \cref{eq:auxsystem:Dual}  with $\dim(\relint (\cG) ) = m+r$.
\begin{enumerate}
\item 
Pick any $0 \ne w \in \Rnp$ with $|\supp(w)| = n-r$.
Let 
\[ 
\{w\}^\perp = \spann  \{d_i\}_{i=1}^{n-1} \subset \Rn \quad \left( = \nul (w^T) \right) . 
\]
We let $D \in \R^{(n-1)\times n}$ be the matrix where its rows consist of $\{d_i^T\}_{i=1}^{n-1}$.
We let $R\in \R^{m\times (n-1)}$ be a random matrix and we set $A = R D$.
We note that $Aw=0$.
\item Pick $s\in \Rnp$ so that 
\[
s_i  =
\left\{
\begin{array}{ll}
0 & \text{if } i\in \supp(w) \\
\text{positive } & \text{if } i \notin \supp(w).
\end{array} 
\right.
\]
We note that $\<w,s\> =0$ holds.
\item Pick $y \in \Rm$ and set $c = A^T y +s$. We note that $\<c,w\> =0$ holds.
\end{enumerate}
For the empirics, we construct the objective function $b^Ty$ of $(\cD)$ by choosing a vector $\hat{x}\in \Rnpp$ and setting $b = A\hat{x}$.

\newpage

\cleardoublepage
\addcontentsline{toc}{section}{Index}
\label{ind:indexpg}
\printindex

\cleardoublepage 




\bibliographystyle{plain}


\bibliography{.haesolArXivJan2023} 


\addcontentsline{toc}{section}{\textbf{References}}


\end{document}